\documentclass[11pt,twoside]{amsart}
\usepackage{amsmath, amsfonts, amsthm, amssymb, verbatim}

\usepackage[mathcal]{euscript}
\usepackage{graphicx}
\usepackage{caption,subfig}
\usepackage{color}
\usepackage{tikz}
\usetikzlibrary{calc}

\usepackage{geometry}
\usepackage[colorlinks=true, pdfstartview=FitV, linkcolor=blue, citecolor=red, urlcolor=blue]{hyperref}
\hypersetup{breaklinks=true}

\newtheorem{theorem}{Theorem}[section]
\newtheorem{definition}[theorem]{Definition}
\newtheorem{proposition}[theorem]{Proposition}
\newtheorem{remark}[theorem]{\it Remark\/}
\newtheorem{assumption}[theorem]{Assumption}
\newtheorem{lemma}[theorem]{Lemma}

\numberwithin{equation}{section}
\numberwithin{figure}{section}

\newcommand{\RR}{\mathbb{R}}
\newcommand{\CC}{\mathbb{C}}
\newcommand{\calA}{\mathcal{A}}

\newcommand{\DD}{\mathbb{D}}
\newcommand{\NN}{\mathbb{N}}
\newcommand{\ZZ}{\mathbb{Z}}

\begin{document}

\title[Numerical boundary layers]{Stability of finite difference schemes\\
for hyperbolic initial boundary value problems:\\
Numerical boundary layers.}

\author[Benjamin Boutin \& Jean-Fran{\c c}ois Coulombel]
{Benjamin Boutin \&  Jean-Fran{\c c}ois Coulombel}

\address{IRMAR (UMR CNRS 6625), Universit\'e de Rennes, Campus de Beaulieu, 35042 Rennes Cedex, France.}
\thanks{Research of B. B. was partially supported by ANR project ACHYLLES, ANR-14-CE25-0001-03.}
\email{Benjamin.Boutin@univ-rennes1.fr}

\address{CNRS, Universit\'e de Nantes, Laboratoire de Math\'ematiques Jean Leray (CNRS UMR6629), 
2 rue de la Houssini\`ere, BP 92208, 44322 Nantes Cedex 3, France.}
\thanks{Research of J.-F. C. was partially supported by ANR project BoND, ANR-13-BS01-0009-01.}
\email{Jean-Francois.Coulombel@univ-nantes.fr}


\date{\today}

\begin{abstract}
In this article, we give a unified theory for constructing boundary layer expansions for discretized transport 
equations with homogeneous Dirichlet boundary conditions. We exhibit a natural assumption on the discretization 
under which the numerical solution can be written approximately as a two-scale boundary layer expansion. 
In particular, this expansion yields discrete semigroup estimates that are compatible with the continuous 
semigroup estimates in the limit where the space and time steps tend to zero. The novelty of our approach 
is to cover numerical schemes with arbitrarily many time levels, while semigroup estimates were restricted, 
up to now, to numerical schemes with two time levels only.
\end{abstract}

\maketitle



\tableofcontents

\section{Introduction and main result}

\subsection{Introduction}

The analysis of numerical boundary conditions for hyperbolic equations is a delicate subject for which several 
definitions of stability can be adopted. Any such definition relies on the choice of a given topology that is a 
discrete analogue of the norm of some functional space in which the underlying continuous problem is 
known to be well-posed. The stability theory for numerical boundary conditions developed in \cite{gks}, 
though rather natural in view of the results of \cite{kreiss2} for partial differential equations, may have suffered 
from its "technicality". As {\sc Trefethen} and {\sc Embree} \cite[chapter 34]{TE} say: \lq \lq [...] {\it the term 
GKS-stable is quite complicated. This is a special definition of stability,} [...], {\it that involves exponential 
decay factors with respect to time and other algebraic terms that remove it significantly from the more familiar 
notion of bounded norms of powers}\rq \rq. More precisely, the definition of stability in \cite{gks} corresponds 
to norms of $\ell^2_{t,x}$ type for the numerical solution ($t$ denotes time and $x$ denotes the space 
variable), while in many problems of evolutionary type one is more used to the $\ell^\infty_t(\ell^2_x)$ topology. 
In terms of operator theory, the definition of stability in \cite{gks} corresponds to {\it resolvent estimates}, while 
the more familiar notion of bounded norms of powers corresponds to {\it semigroup estimates}. Hence a natural 
-though delicate- question in the theory of hyperbolic boundary value problems is to pass from GKS type (that 
is, resolvent) estimates to semigroup estimates. In the context of partial differential equations, this problem 
has received a somehow final answer in \cite{metivier2}, see references therein for historical comments on this 
problem. In the context of numerical schemes, the derivation of semigroup estimates is not as well understood 
as for partial differential equations. Semigroup estimates have been derived in \cite{wu} for discrete scalar 
equations, and in \cite{jfcag} for systems of equations. However, the analysis in \cite{wu} and \cite{jfcag} only 
deals with schemes with two time levels, and does not extend as such to schemes with three or more time 
levels (e.g., the leap-frog scheme).

In this article, we focus on Dirichlet boundary conditions and derive semigroup estimates for a class of numerical 
schemes with arbitrarily many time levels. The reasons why we choose Dirichlet boundary conditions are twofold. 
First, these are the only boundary conditions for which, independently of the (stable) numerical scheme that is 
used for discretizing a scalar transport equation, stability in the sense of GKS is known to hold. The latter result 
dates back to \cite{goldberg-tadmor} and is recalled later on. Second, homogeneous Dirichlet boundary conditions 
typically give rise to numerical boundary layers and therefore to an accurate description of the numerical solution 
by means of a two-scale expansion. We combine these two favorable aspects of the Dirichlet boundary conditions 
in our derivation of a semigroup estimate.

The study of numerical boundary layers has received much attention in the past decades, including for nonlinear 
systems of conservation laws, see for instance \cite{DuboisLeFloch,GisclonSerre,chainais-grenier}. As far as 
we know, all previous studies have considered numerical schemes with a three point stencil and two time levels. 
In this article, we focus on linear transport equations and exhibit a class of numerical schemes for which the 
homogeneous Dirichlet boundary conditions give rise to numerical boundary layers. The stencil can be arbitrarily 
wide. As follows from our criterion, the occurrence of boundary layers is not linked with the order of accuracy of 
the numerical scheme, which is a {\it low frequency} property, but rather with its {\it high frequency} behavior. For 
instance, the Lax-Wendroff discretization displays numerical boundary layers when combined with Dirichlet boundary 
conditions (and such layers have the same width as for the Lax-Friedrichs scheme) but the leap-frog scheme 
does not\footnote{The leap-frog scheme rather generates incoming highly oscillating wave packets, as explained 
at the end of this article.}, though both Lax-Wendroff and leap-frog schemes are formally of order $2$.

\subsection{Notations}

We consider a one-dimensional scalar transport equation: 
\begin{equation}
\label{eq:transport}
\partial_t u + a \, \partial_x u = 0,\quad t>0 \, , \, x>0 \, ,
\end{equation}
where the velocity is $a \neq 0$. The transport equation \eqref{eq:transport} is supplemented with an initial 
condition $u_0$ that belongs to a functional space that is made precise later on. In the case $a>0$, that is, 
if we consider an {\it incoming} transport equation, we also supplement \eqref{eq:transport} with homogeneous 
Dirichlet boundary condition:
\begin{equation}
\label{eq:Dirichlet}
u(0,t) = 0,\quad t>0 \, .
\end{equation}

The finite difference scheme under consideration is assumed to be obtained by the so-called {\it method of lines}, 
see, e.g., \cite{gko}. In other words, we start with \eqref{eq:transport} and first use a space discretization. The latter 
is supposed to be linear with $r$ points on the left and $p$ points on the right. In other words, we consider some 
coefficients $a_{-r},\dots,a_p$, where $p,r$ are fixed nonnegative integers, together with a space step $\Delta x>0$, 
and approximate \eqref{eq:transport} by the system of ordinary differential equations:
\begin{equation}
\label{semidiscret}
\dot{u}_j +\dfrac{1}{\Delta x} \, \sum_{\ell=-r}^p a_\ell \, u_{j+\ell} =0 \, , 
\end{equation}
where $u_j(t)$ represents an approximation of the solution $u$ to \eqref{eq:transport} in the neighborhood of 
the point $x_j :=j \, \Delta x$. The integers $r,p$ are fixed by assuming $a_{-r}\neq 0$ and $a_p\neq 0$. The latter 
system of ordinary differential equations is then approximated by means of a (possibly multistep) explicit numerical 
integration method. We refer to \cite{hnw,hw} for an extensive study of numerical methods for ordinary differential 
equations. Applying a linear explicit multistep method to \eqref{semidiscret} yields the numerical approximation
\begin{equation}
\label{eq:schemelinear}
\sum_{\sigma=0}^k \alpha_\sigma \, u_j^{n+\sigma} + \lambda \, \sum_{\sigma=0}^{k-1} \beta_\sigma \, 
\sum_{\ell=-r}^p a_\ell \, u_{j+\ell}^{n+\sigma} = 0 \, .
\end{equation}
with $k \ge 1$ and fixed constants $\alpha_0,\dots,\alpha_k,\beta_0,\dots,\beta_{k-1}$. The multistep integration 
method is normalized by assuming $|\alpha_0|+|\beta_0|>0$ and $\alpha_k =1$. In \eqref{eq:schemelinear}, we 
have made use of the notation $\lambda:=\Delta t/\Delta x$ for the so-called Courant-Friedrichs-Lewy parameter. 
In what follows, the parameter $\lambda$ is kept fixed\footnote{This assumption could be weakened by assuming 
that the ratio $|a| \, \Delta t/\Delta x$ is bounded from below and from above, but we shall restrict to the more 
common case where the ratio is fixed for simplicity.}, and we consider the space and time grid $x_j := j\, \Delta x$, 
$t^n := n\, \Delta t$ for $j,n \in\NN$. For notational convenience, we introduce the (dimensionless) constant 
$\tau>0$ that satisfies
\begin{equation}
\label{eq:CFLnumber}
\Delta x = \tau \, |a| \, \Delta t \, .
\end{equation}
We keep $\Delta t \in (0,1]$ as the only small parameter and $\Delta x \in (0,1/\lambda]$ varies accordingly.

Since we are approximating the transport equation \eqref{eq:transport} on the half-line $\RR^+$, the space grid 
is indexed by $\NN$. This means that the numerical approximation \eqref{eq:schemelinear} takes place for 
$j \ge r$. We then supplement \eqref{eq:schemelinear} with homogeneous Dirichlet boundary conditions on 
the "numerical" boundary:
\begin{equation}
\label{eq:homogDirichlet}
u_j^n = 0, \quad 0\leq j \leq r-1,\quad n\geq k \, ,
\end{equation}
independently of the sign of $a$. The scheme \eqref{eq:schemelinear}, \eqref{eq:homogDirichlet} is ignited by 
$k$ initial data, which correspond to the approximation of the solution to \eqref{eq:transport} at times $t^0,\dots,
t^{k-1}$. For simplicity, we assume that the initial data for \eqref{eq:schemelinear}, \eqref{eq:homogDirichlet} 
are given by the standard piecewise constant approximation of the exact solution to \eqref{eq:transport}. In 
other words, we set:
\begin{equation}
\label{eq:initialdata}
u_j^n := \dfrac{1}{\Delta x} \, \int_{x_j}^{x_{j+1}} u_0(x-a\, t^n) \, {\rm d}x \, ,\quad j \ge 0 \, ,\quad 
n=0,\dots,k-1 \, ,
\end{equation}
where the initial condition $u_0$ for  \eqref{eq:transport} has been extended by zero to $\RR^-$ in the case $a>0$.

The following two assumptions are the minimal consistency requirements for the numerical scheme 
\eqref{eq:schemelinear}.

\begin{assumption}[Consistency of the space discretization]
\label{as:consistency}
The coefficients $a_{-r},\dots,a_p$ in \eqref{eq:schemelinear} satisfy
\begin{gather}
\sum_{\ell=-r}^p a_\ell = 0 \, ,\label{eq:consistency0}\\
\sum_{\ell=-r}^p \ell \, a_\ell = a \, .\label{eq:consistency1}
\end{gather}
\end{assumption}

\begin{assumption}[Consistency of the linear multistep integration method]
\label{as:consistencymultistep}
The coefficients $\alpha_0,\dots,\alpha_k$, $\beta_0,\dots,\beta_{k-1}$ of the time integration method in 
\eqref{eq:schemelinear} satisfy
\begin{equation*}
\sum_{\sigma=0}^k \alpha_\sigma = 0 \, ,\quad 
\sum_{\sigma=0}^k \sigma \, \alpha_\sigma =\sum_{\sigma=0}^{k-1} \beta_\sigma \, .
\end{equation*}
\end{assumption}

In the case $k=1$, that is for numerical schemes with two time levels, the normalization gives 
$\alpha_0=\alpha_1=\beta_0=1$, and \eqref{eq:schemelinear} reduces to the standard form
\begin{equation*}
u_j^{n+1} -u_j^n +\lambda \, \sum_{\ell=-r}^p a_\ell \, u_{j+\ell}^n = 0 \, .
\end{equation*}
If $p=r=1$, we obtain the class of three point schemes that encompasses both the Lax-Friedrichs and 
Lax-Wendroff scheme.

As a direct consequence of the first consistency condition~\eqref{eq:consistency0}, it appears that the 
scheme~\eqref{eq:schemelinear} admits a conservative form in the following sense. There exists a linear 
numerical flux function $F$ with real coefficients:
\begin{equation*}
F(v_j,\ldots,v_{j+p+r-1}) := \sum_{\ell=-r}^{p-1} f_\ell \, v_{j+\ell+r} \, ,
\end{equation*}
such that
\begin{equation}
\label{eq:flux}
\sum_{\ell=-r}^p a_\ell \, u_{j+\ell} =F(u_{j-r+1},\dots,u_{j+p})-F(u_{j-r},\dots,u_{j+p-1}) \, .
\end{equation}
In particular, \eqref{eq:schemelinear} also takes the conservative form
\begin{equation}
\label{eq:schemecons}
\sum_{\sigma=0}^k \alpha_\sigma \, u_j^{n+\sigma} + \lambda \, \sum_{\sigma=0}^{k-1} \beta_\sigma \, 
\Big( F(u_{j-r+1}^{n+\sigma},\ldots,u_{j+p}^{n+\sigma}) -F(u_{j-r}^{n+\sigma},\ldots,u_{j+p-1}^{n+\sigma}) 
\Big) = 0 \, .
\end{equation}
From the second consistency condition~\eqref{eq:consistency1}, it follows that $F(u,\ldots,u)=a\, u$ for any 
$u \in \RR$. This is the usual consistency property of $F$ with the exact flux $(u\mapsto a\, u)$ of the transport 
equation \eqref{eq:transport} written as a conservation law.\\

Our final assumption is the standard $\ell^2$-stability assumption for \eqref{eq:schemelinear} when the scheme 
is considered on the whole real line $j \in \ZZ$:

\begin{assumption}[Stability for the Cauchy problem]
\label{as:stability}
There exists a constant $C>0$ such that, for all $\Delta t \in (0,1]$, the solution to
\begin{equation*}
\sum_{\sigma=0}^k \alpha_\sigma \, u_j^{n+\sigma} + \lambda \, \sum_{\sigma=0}^{k-1} \beta_\sigma \, 
\sum_{\ell=-r}^p a_\ell \, u_{j+\ell}^{n+\sigma} = 0 \, ,\quad j \in \ZZ \, ,\quad n \in \NN \, ,
\end{equation*}
satisties
\begin{equation*}
\sup_{n \in \NN} \, \sum_{j \in \ZZ} \Delta x \, |u_j^n|^2 \le C \, \sum_{\sigma =0}^{k-1} \, \sum_{j \in \ZZ} 
\Delta x \, |u_j^\sigma|^2 \, .
\end{equation*}
\end{assumption}

As is well-known, assumption~\ref{as:stability} can be rephrased thanks to Fourier analysis. More precisely, 
if we introduce the function $\calA$ defined by:
\begin{equation}
\label{eq:amplification}
\forall z\in\CC\setminus\{0\}, \quad \calA(z) = \sum_{\ell=-r}^p a_\ell \, z^\ell \, ,
\end{equation}
then applying the Fourier transform to \eqref{eq:schemelinear} yields for all $\xi\in\RR$:
\begin{equation*}
\sum_{\sigma=0}^k \alpha_\sigma \, \widehat{u^{n+\sigma}}(\xi) +\lambda \, \sum_{\sigma=0}^{k-1} 
\beta_\sigma \, \calA({\rm e}^{i\, \Delta x \, \xi}) \, \widehat{u^{n+\sigma}}(\xi) =0 \, ,
\end{equation*}
where $u^n$ is the piecewise constant function that takes the value $u_j^n$ on the cell $[j \, \Delta x,(j+1) \, 
\Delta x)$. The stability assumption~\ref{as:stability} is equivalent to requiring that there exists a constant $C>0$ 
such that for all $\eta \in \RR$, and for all given $x_0,\dots,x_{k-1} \in \CC$, the solution $(x_\sigma)_{\sigma \in \NN}$ 
to the recurrence relation
\begin{equation*}
\forall \, n \in \NN \, ,\quad 
\sum_{\sigma=0}^k \alpha_\sigma \, x_{n+\sigma} +\lambda \, \sum_{\sigma=0}^{k-1} 
\beta_\sigma \, \calA({\rm e}^{i\, \eta}) \, x_{n+\sigma} =0 \, ,
\end{equation*}
satisfies
\begin{equation*}
\sup_{n \in \NN} |x_n|^2 \le C\, \Big( |x_0|^2 +\cdots +|x_{k-1}|^2 \Big) \, .
\end{equation*}
In particular, the closed curve $\{ -\lambda \, \calA({\rm e}^{i\, \eta}) \, , \, \eta \in \RR \}$ should be contained in 
the so-called stability region of the numerical integration method, see \cite[Definition V.1.1]{hw}. Observe now 
that the consistency assumption~\ref{as:consistency} introduced above can be rewritten under the form:
\begin{equation}
\label{eq:Aconsistency}
\calA(1)=0 \quad \textrm{and} \quad \calA'(1) = a \neq 0 \, .
\end{equation}
Since $\calA$ vanishes at $1$, $0$ should belong to the stability region of the numerical integration method, 
which implies (see \cite[Chapter III.3]{hnw}):
\begin{equation}
\label{conditionbeta}
\sum_{\sigma=0}^k \sigma \, \alpha_\sigma \neq 0 \, .
\end{equation}

\begin{remark}
In the case $k=1$, the stability assumption~\ref{as:stability} is equivalent to:
\begin{equation}
\label{eq:Astability}
\forall \, z \in \mathbb{S}^1 \, ,\quad |1-\lambda \, \calA(z)| \leq 1 \, .
\end{equation}
In particular, assumption~\ref{as:stability} constraints the CFL number $\lambda$ to be "small enough", and 
$\calA(z)$ can not be a negative number.
\end{remark}

\subsection{Main result}

The main result of this paper is the following theorem.

\begin{theorem}[Semigroup estimate]
\label{thm:stab1step}
Consider a linear scheme of the form~\eqref{eq:schemelinear} satisfying the consistency assumptions~\ref{as:consistency} 
and \ref{as:consistencymultistep}, the stability assumption~\ref{as:stability} and the "dissipative" assumption~\ref{as:circle} 
introduced later on. Consider an initial condition $u_0\in H^2(\RR_+)$ for \eqref{eq:transport} such that 
\begin{equation*}
\begin{cases}
u_0(0)=0 \, ,&\text{\rm if } a<0 \, ,\\
u_0(0)=u_0'(0)=0 \, ,&\text{\rm if } a>0 \, .
\end{cases}
\end{equation*}

Let $T>0$ and, for $\Delta t\in(0,1]$, let us define $N_T$ as the largest integer such that $\Delta t \, N_T\leq T$. 
Let also $\mu \in [0,1/3]$. Then there exists a constant $C>0$, that is independent of $T,\Delta t,\mu,u_0$ such that 
the solution $(u_j^n)_{j\geq 0, n \geq 0}$ to \eqref{eq:schemelinear}-\eqref{eq:homogDirichlet}-\eqref{eq:initialdata} 
satisfies
\begin{equation}
\label{eq:semigroupestimate}
\sup_{n \leq N_T} \, \sum_{j \geq 0} \Delta x \, |u_j^n|^2 \leq C \, \Big( \|u_0\|^2_{L^2(\RR^+)} +\Delta t^{1-3\, \mu} \, 
{\rm e}^{2\, T \, \Delta t^\mu} \, \| u_0 \|^2_{H^2(\RR^+)} \Big) \, .
\end{equation}
\end{theorem}

Let us observe that \eqref{eq:semigroupestimate} is compatible with the "continuous" estimate
\begin{equation*}
\sup_{t \ge 0} \, \| u(t) \|_{L^2(\RR^+)}^2 \le C \, \|u_0\|^2_{L^2(\RR^+)} \, ,
\end{equation*}
as $\Delta t$ tends to zero. The role of assumption~\ref{as:circle} is to derive a boundary layer expansion for 
$(u_j^n)_{j\geq 0, n \geq 0}$, that is to decompose $(u_j^n)$ as in \cite{DuboisLeFloch,GisclonSerre,chainais-grenier} 
under the form
\begin{equation*}
u_j^n \sim u^{\rm int}(x_j,t^n) +u^{\rm bl}(j,t^n) \, ,
\end{equation*}
where the {\it boundary layer profile} $u^{\rm bl}$ depends on the "fast" variable $j=x_j/\Delta x$ and has exponential 
decay at infinity, while the {\it interior profile} $u^{\rm int}$ depends on the "slow" variable $x_j$. As follows from the 
analysis below, the derivation of such two-scale expansions is not linked to any viscous behavior of \eqref{eq:schemelinear} 
(as the scaling $x_j/\Delta x$ might suggest at first glance).

The parameter $\mu$ can diminish the $T$-dependence of the constants in \eqref{eq:semigroupestimate}. 
In particular, given any $\varepsilon>0$ and $T>0$, there holds
\begin{equation*}
\sup_{n \leq N_T} \, \sum_{j \geq 0} \Delta x \, |u_j^n|^2 \leq C \, \Big( \|u_0\|^2_{L^2(\RR^+)} 
+2\, \Delta t^{1-\varepsilon} \, \| u_0 \|^2_{H^2(\RR^+)} \Big) \, ,
\end{equation*}
for $\Delta t$ sufficiently small (depending on $T$).

Section \ref{sect2} is devoted to the construction of boundary layer expansions for solutions to \eqref{eq:schemelinear}, 
\eqref{eq:homogDirichlet}. Theorem \ref{thm:stab1step} is proved in Section \ref{sect3} by means of a careful 
error analysis. We discuss some examples in Section \ref{sect4} together with the relevance of 
assumption~\ref{as:circle}.

\section{Numerical boundary layers}
\label{sect2}

\subsection{Formal derivation of the boundary layer expansion}

Our first goal is to understand when the numerical solution $(u^n_j)_{j\geq 0, n\geq 0}$ of the scheme 
\eqref{eq:schemelinear}, \eqref{eq:homogDirichlet} can be approximated by an asymptotic boundary 
layer expansion:
\[
u^{\rm app}_{j,n} : = u^{\rm int}(x_j,t^n) + u^{\rm bl}(j,t^n) \, ,\quad j\geq 0\, , \, n\geq 0\, .
\]
In the latter decomposition, we expect $u^{\rm bl}$ to have fast decay at infinity. The functions $u^{\rm int}$ 
and $u^{\rm bl}$ are to be defined in such a way that $(u^{\rm app}_{j,n})$ represents an accurate approximation 
of $(u_j^n)$ as $\Delta t$ tends to $0$. Roughly speaking, the term $u^{\rm int}$ takes care of the interior behavior 
of the solution far from the boundary, and $u^{\rm bl}$ involves the boundary layer correction that is localized 
in a neighborhood of $x=0$ and matches the boundary conditions \eqref{eq:homogDirichlet}.

We shall force the approximate solution to satisfy the initial conditions:
\begin{equation}
\label{eq:approxInitial}
u^{\rm app}_{j,n} = u^n_j \, ,\quad j\geq 0 \, , \, n=0,\dots,k-1 \, .
\end{equation}
In this way, the error $(u^{\rm app}_{j,n} -u^n_j)$ will satisfy a recurrence relation of the form \eqref{eq:schemelinear}, 
\eqref{eq:homogDirichlet} with "small" source terms but will have zero initial data. We also expect the approximate 
solution to satisfy \eqref{eq:homogDirichlet}, or rather
\begin{equation}
\label{eq:approxDirichlet}
u^{\rm app}_{j,n} \simeq 0 \, ,\quad 0\leq j \leq r-1 \, ,\quad n\geq k \, ,
\end{equation}
where, by $\simeq 0$, we mean for instance that $u^{\rm app}_{j,n}$ should be $O(\Delta t)$ on the boundary.

For technical reasons that will be made precise in Section \ref{sect3}, we shall define the boundary layer 
term through a two term expansion of the form:
\[
u^{\rm bl}(j,t^n) : = u^{\rm bl, 0}(j,t^n) + \Delta x \, u^{\rm bl, 1}(j,t^n),
\]
involving a zero order term $u^{\rm bl, 0}$ plus a first order corrector $u^{\rm bl, 1}$ that will be used to remove 
part of the consistency error.
\medskip

We follow the discussions in \cite{DuboisLeFloch,GisclonSerre,chainais-grenier} and briefly present hereafter 
a schematic derivation of the equations that will govern the three sequences $u^{\rm int}$, $u^{\rm bl,0}$ and 
$u^{\rm bl,1}$. To that aim, let us introduce the following consistency error:
\begin{equation*}
\varepsilon_{j,n+k} := \dfrac{1}{\Delta t} \, \left(\sum_{\sigma=0}^k \alpha_\sigma \, u^{\rm app}_{j,n+\sigma} 
+\lambda \, \sum_{\sigma=0}^{k-1} \beta_\sigma \, \sum_{\ell=-r}^p a_\ell \, u^{\rm app}_{j+\ell,n+\sigma} \right) \, ,
\end{equation*}
with  $j\geq r$, and $n\geq 0$.

\begin{itemize}
\item 
At a fixed positive distance from the boundary, the limit $\Delta t \to 0$ corresponds to $j \to +\infty$ and the 
boundary layer term $u^{\rm bl}$ becomes negligible with respect to $u^{\rm int}$. The above consistency error 
reads (up to smaller terms)
\[
\varepsilon_{j,n+k} \simeq \dfrac{1}{\Delta t} \, \left(\sum_{\sigma=0}^k \alpha_\sigma \, u^{\rm int}_{j,n+\sigma} 
+\lambda \, \sum_{\sigma=0}^{k-1} \beta_\sigma \, \sum_{\ell=-r}^p a_\ell \, u^{\rm int}_{j+\ell,n+\sigma} \right) \, .
\]
This quantity will be of order $O(\Delta t)$ provided that $u^{\rm int}$ is a smooth solution to the continuous 
equation~\eqref{eq:transport} (recall the consistency assumptions of the numerical scheme \eqref{eq:schemelinear}).\\

\item 
Close to the boundary, that is for a fixed index $j \ge r$, the limit $\Delta t \to 0$ makes $x_j$ tend to zero. 
If the interior solution $u^{\rm int}$ is smooth enough, we get (recall that $j$ is fixed) $u^{\rm int}(x_j,t^n) 
= u^{\rm int}(0,t^n) +O(\Delta t)$ and $u^{\rm int}(0,t^{n+\sigma}) = u^{\rm int}(0,t^n) +O(\Delta t)$. Then 
the consistency error reads\footnote{Here we use the consistency conditions for the coefficients in 
\eqref{eq:schemelinear}.} (up to $O(1)$ terms):
\begin{equation*}
\varepsilon_{j,n+k} \simeq \dfrac{1}{\Delta t} \, \left(\sum_{\sigma=0}^k \alpha_\sigma \, u^{\rm bl,0} (j,t^{n+\sigma}) 
+\lambda \, \sum_{\sigma=0}^{k-1} \beta_\sigma \, \sum_{\ell=-r}^p a_\ell \, u^{\rm bl,0} (j+\ell,t^{n+\sigma}) \right) \, .
\end{equation*}
Due to the consistency of the numerical integration method, and assuming that $u^{\rm bl,0}$ depends smoothly 
enough on the time variable, we get
\begin{equation*}
\varepsilon_{j,n+k} \simeq \dfrac{1}{\Delta x} \, \left( \sum_{\sigma=0}^{k-1} \beta_\sigma \right) \, 
\sum_{\ell=-r}^p a_\ell \, u^{\rm bl,0} (j+\ell,t^{n+k}) \, ,
\end{equation*}
The first boundary layer profile $u^{\rm bl,0}$ needs therefore to satisfy the recurrence relation\footnote{Recall that 
by our consistency and stability assumptions, the sum of the $\beta_\sigma$ is nonzero.}:
\begin{equation}
\label{eq:zerothBL}
\sum_{\ell=-r}^p a_\ell \, u^{\rm bl,0}(j+\ell,t^{n})  = 0 \, , \quad j\geq r \, ,\, n\geq k \, .
\end{equation}
In terms of flux quantities, the relation \eqref{eq:zerothBL} corresponds to requiring
$$
\sum_{\ell = -r}^{p-1} f_\ell \, u^{\rm bl, 0}(j+\ell+r,t^n) \equiv \text{\rm C}^{\rm st} \, ,
$$
with an integration constant that only depends on $n$, but not on $j$. The constant is easily seen to be 
zero due to the required behavior of the boundary layer profiles at infinity. In addition, the boundary condition 
\eqref{eq:approxDirichlet} imposes (up to an $O(\Delta t)$ term) the trace of $u^{\rm bl,0}$ on the numerical 
boundary:
\begin{equation}
\label{eq:zerothDirichlet}
u^{\rm bl,0}(j,t^n) = -u^{\rm int}(0,t^n)\, , \quad 0\leq j \leq r-1 \, ,\ n\geq 0 \, .
\end{equation}

\item 
We still keep the index $j$ fixed and expand the consistency error at the following order with respect to $\Delta t$. 
Assuming that $u^{\rm int}$ is smooth enough so that its associated consistency error is $O(\Delta t)$ up to the 
boundary, the overall consistency error reads (up to $O(\Delta t)$ terms):
\[
\varepsilon_{j,n+k} \simeq \dfrac{1}{\Delta t} \, \sum_{\sigma=0}^k \alpha_\sigma \, u^{\rm bl,0} (j,t^{n+\sigma}) 
+\left( \sum_{\sigma=0}^{k-1} \beta_\sigma \right) \, \sum_{\ell=-r}^p a_\ell \, u^{\rm bl,1} (j+\ell,t^n) \, .
\]
We then require the first boundary layer corrector $u^{\rm bl,1}$ to satisfy:
\begin{equation}
\label{eq:firstBL}
\sum_{\ell=-r}^p a_\ell \, u^{\rm bl,1} (j+\ell,t^n) +\dfrac{1}{\Delta t} \, 
\left( \sum_{\sigma=0}^{k-1} \beta_\sigma \right)^{-1} \, 
\sum_{\sigma=0}^k \alpha_\sigma \, u^{\rm bl,0} (j,t^{n+\sigma}) =0 \, , \quad j \geq r \, .
\end{equation}
Since our analysis considers numerical schemes of order $1$ or higher, the precise value of $u^{\rm bl,1}$ 
on the numerical boundary does little matter since any other choice than the one below will introduce a new 
$O(\Delta t)$ error that will just have the same order as the interior consistency error. For simplicity, we therefore 
require $u^{\rm bl ,1}$ to satisfy;
\begin{equation}
\label{eq:firstDirichlet}
u^{\rm bl,1}(j,t^n) = 0 \, , \quad 0\leq j \leq r-1 \, ,\ n\geq 0 \, .
\end{equation}
\end{itemize}

The above formal derivation of the profile equations \eqref{eq:zerothBL} and \eqref{eq:firstBL} motivates the 
analysis of the recurrence relation \eqref{eq:zerothBL}. More precisely, we are going to determine the solutions 
to \eqref{eq:zerothBL} that tend to zero at infinity. The precise definition of the approximate solution $u^{\rm app}$ 
is given in subsection \ref{defuapp}.

\subsection{A preliminary result}

Let us recall that the function $\calA$, which is linked to the amplification matrix for the scheme \eqref{eq:schemelinear}, 
is defined in \eqref{eq:amplification}. The consistency assumption~\ref{as:consistency} implies that $1$ is a simple root 
of $\calA$. The following assumption will turn out to be crucial in the forthcoming analysis.

\begin{assumption}
\label{as:circle}
The value $z=1$ is the unique root of $\calA$ on $\mathbb{S}^1$: 
\begin{equation*}
\forall \, \theta \in [-\pi,\pi] \setminus \{ 0 \} \, ,\quad \calA({\rm e}^{i\, \theta}) \neq 0 \, .
\end{equation*}
\end{assumption}

\begin{remark}
In the case $k=1$, assumption~\ref{as:circle} is obviously satisfied for every dissipative scheme (for which we recall 
that there exist $c>0$ and $k \in \NN^*$ such that for all $|\theta| \leq \pi$, $|1-\lambda \, \calA({\rm e}^{i\, \theta})| 
\leq 1 - c\, \theta^{2\, k}$). However, we underline at this level that some non-dissipative schemes satisfy 
assumption~\ref{as:circle} too, e.g. the Lax-Friedrichs scheme (that is considered in \cite{chainais-grenier}) 
for which $\calA({\rm e}^{i\, \theta})=\cos \theta -1 -i\, \lambda \, a \, \sin \theta$).
\end{remark}

The main result of this subsection is the following Lemma.

\begin{lemma}
\label{lm:roots}
Under Assumptions~\ref{as:consistency}, \ref{as:consistencymultistep}, \ref{as:stability} and~\ref{as:circle}, the equation 
$\calA(z)=1$ admits exactly $R$ roots (with multiplicity) in $\DD \setminus \{ 0 \} = \{z\in\CC \, , \, 0<|z|< 1\}$ where
\begin{equation*}
R=\begin{cases}
r,&\textrm{if } a<0 \, ,\\
r-1,&\textrm{if } a>0 \, .
\end{cases}
\end{equation*}
\end{lemma}

\noindent Let us observe that in the case $a>0$, $r$ can not be zero and therefore one gets a nonnegative integer 
for $R$. Indeed, the value $r=0$ is prohibited by the fact that the numerical dependence domain would not include 
the "continuous" dependence domain, see \cite{cfl}.

The proof of Lemma \ref{lm:roots} makes use of the following simple observation which we have not found in 
\cite{hw} and therefore state here. We keep the notations of \cite[Chapter III.2]{hnw}.

\begin{lemma}
\label{lm:multipas}
Consider an ordinary differential equation of the form $\dot{y}=f(y)$, and the explicit linear multistep integration method:
\begin{equation}
\label{multipas}
\sum_{\sigma=0}^k \alpha_\sigma \, y_{n+\sigma} =\Delta t \, \sum_{\sigma=0}^{k-1} \beta_\sigma \, f_{n+\sigma} 
\, ,
\end{equation}
with the normalization $\alpha_k=1$, $|\alpha_0|+|\beta_0|>0$. Assume that the method is stable 
(in the sense of \cite[Definition III.3.2]{hnw}) and that it is of order $1$ or higher. Then the stability 
region for this method contains no positive real number.
\end{lemma}

\begin{proof}
Following \cite{hnw,hw}, we introduce the polynomials
\begin{equation*}
\varrho(X) := \sum_{j=0}^k \alpha_j \, X^j \, ,\quad 
\sigma(X) := \sum_{j=0}^{k-1} \beta_j \, X^j \, .
\end{equation*}
The assumptions of Lemma \ref{lm:multipas} can be rephrased as:
\begin{equation*}
\varrho(1)=0 \, ,\quad \varrho'(1) =\sigma(1) \neq 0 \, ,
\end{equation*}
and $\varrho$ has no root of $z$ satisfying $|z|>1$. In particular, $\varrho'(1)$ must be positive for otherwise 
(recall $\alpha_k=1$) $\varrho$ would have a real root in the open interval $(1,+\infty)$. We therefore have 
$\sigma(1)>0$.

For any given $\mu>0$, the real polynomial:
\begin{equation*}
P_\mu(X) :=\varrho(X) -\mu \, \sigma(X) \, ,
\end{equation*}
has degree $k$ and is unitary. It tends to $+\infty$ at $+\infty$ and $P_\mu(1)=-\mu \, \sigma(1)<0$. Hence 
$P_\mu$ vanishes in the open interval $(1,+\infty)$ and $\mu$ does not belong to the stability region of the 
numerical method.
\end{proof}

\noindent Lemma \ref{lm:multipas} is consistent with the plots in \cite{hw} of the stability regions for the explicit 
Adams and Nystr\"om methods. Observe however that some stability regions may contain complex numbers of 
positive real part, e. g., the explicit Adams method of order $3$.

\begin{proof}[Proof of Lemma \ref{lm:roots}]
Under assumption~\ref{as:circle}, $\calA$ has no other zero on $\mathbb{S}^1$ than $z=1$ (with multiplicity 
$1$). On the other hand, $\calA$ admits a unique pole over $\CC$, at $z=0$ and of order $r$ (because we have 
$a_{-r} \neq 0$). The cornerstone of the forthcoming proof is the residue theorem for meromorphic functions. 
Being given $\Gamma$ a direct closed complex contour encircling the origin once and on which $\calA$ does 
not vanish, then
\begin{equation}
\label{eq:rouche}
\dfrac{1}{2\, i\, \pi} \, \int_\Gamma \dfrac{\calA'(z)}{\calA(z)} \, {\rm d}z = \# \{ \textrm{zeros inside }\Gamma \} 
- \# \{ \textrm{poles inside }\Gamma \} \, ,
\end{equation}
where zeros and poles are counted with multiplicity. The second integer on the right hand side equals $r$, and 
we intend now to compute $R_\Gamma := \# \{ \textrm{zeros inside }\Gamma \}$ thanks to an appropriate choice 
for the contour~$\Gamma$ (for which $R_\Gamma=R$).\\

\paragraph{\it The contour $\Gamma_\varepsilon$.}
Let us consider some parameter $\varepsilon \in (0,\pi/4]$ sufficiently small (to be determined later on), and 
let us define the contour $\Gamma_{\varepsilon}$ as $\mathbb{S}^1$ but for a small chord avoiding $1$, 
see Figure~\ref{fig:Gamma}. More precisely, we consider the path $\Gamma_\varepsilon$ as the union 
$\Gamma_{\varepsilon,1} \cup \Gamma_{\varepsilon,2}$, with:
$$
\Gamma_{\varepsilon,1} := \Big\{ {\rm e}^{i\, \theta} \, , \, \theta \in [\varepsilon, 2\, \pi -\varepsilon ] \Big\} \, ,\quad 
\Gamma_{\varepsilon,2} :=\Big\{ \cos \varepsilon +i\, \omega \, , \, \omega \in [-\sin \varepsilon, \sin \varepsilon] 
\Big\} \, .
$$

\begin{figure}[h!]
\begin{tikzpicture}[scale=2,>=latex]
\pgfmathsetmacro{\eps}{0.4}
\pgfmathsetmacro{\thet}{ {asin(\eps)} }
\pgfmathsetmacro{\thetm}{ {360-2*asin(\eps/2)} }
\draw[->] (-1.25,0) -- (1.75,0) node[below] {$x$};
\draw[->] (0,-1.25)--(0,1.25) node[right] {$y$};
\draw[thick] ($(0,0) + (\thet:1)$) arc (\thet:\thetm:1);
\node[below] at (1,0) {$1$};
\draw[dashed] (0,0) -- ({1.5*cos(\thet)},{1.5*sin(\thet)});
\draw[->] ($(0,0) + (0:1.5)$) arc (0:\thet:1.5);
\draw[thick] ({cos(\thet)},{-sin(\thet)})--({cos(\thet)},{sin(\thet)});
\node[right] at ({1.5*cos(\thet/2)},{1.5*sin(\thet/2)}) {$\varepsilon$};
\node[left] at ({1.1*cos(135)},{1.1*sin(135)}) {$\Gamma_{\varepsilon,1}$};
\node[left] at ({0.9*cos(\thet/2)},{-0.9*sin(\thet/2)}) {$\Gamma_{\varepsilon,2}$};
\end{tikzpicture}
\caption{The integration contour $\Gamma_\varepsilon$.}
\label{fig:Gamma}
\end{figure}
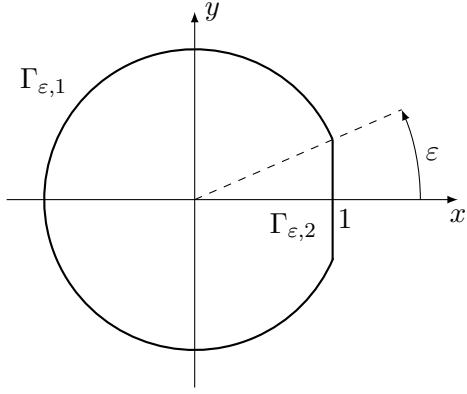

\paragraph{\it Choice of the parameter $\varepsilon$.}
Let us first observe that $1$ is a simple zero of $\calA$ so we can choose $\varepsilon_0>0$ small enough 
such that, for any $\varepsilon \in (0,\varepsilon_0]$, the number of zeros of $\calA$ inside $\Gamma_\varepsilon$ 
equals the number of zeros of $\calA$ in $\DD \setminus \{ 0 \}$.

Our goal now is to show that, for any sufficiently small $\varepsilon>0$, there holds
\begin{equation*}
\begin{cases}
\mp \, \Im \, \calA ({\rm e}^{\pm \, i \, \varepsilon}) >0 \, ,&\text{\rm if} \quadÊa<0 \, ,\\
\pm \, \Im \, \calA ({\rm e}^{\pm \, i \, \varepsilon}) >0 \, ,&\text{\rm if} \quad a>0 \, ,
\end{cases}
\end{equation*}
and for all $z \in \Gamma_{\varepsilon,2}$, $a\, \calA(z) \not \in \RR^+$. These properties follow from studying 
the variation of the function $\Im \, \calA (\cos \varepsilon +i\, \omega)$. Namely, we compute
$$
\dfrac{{\rm d}}{{\rm d}\omega} \, \Im \, \calA (\cos \varepsilon +i\, \omega) =
\Re \, \calA' (\cos \varepsilon +i\, \omega) =a +\Re \, (\calA' (\cos \varepsilon +i\, \omega)-\calA'(1)) \, .
$$
Consequently, if we assume $a>0$, then $(\omega \mapsto \Im \, \calA (\cos \varepsilon +i\, \omega))$ 
is increasing on $[-\sin \varepsilon, \sin \varepsilon]$, while if we assume $a<0$, then $(\omega \mapsto 
\Im \, \calA (\cos \varepsilon +i\, \omega))$ is decreasing on $[-\sin \varepsilon, \sin \varepsilon]$. In any 
of these two cases, $\calA (\cos \varepsilon +i\, \omega)$ is real for at most one value of $\omega$.

We now observe that $\calA (\cos \varepsilon)$ is real. In particular, for any $0<|\omega| \le \sin \varepsilon$, 
$\calA (\cos \varepsilon +i\, \omega)$ belongs to $\CC \setminus \RR$ and the sign property for $\Im \, \calA 
({\rm e}^{\pm \, i \, \varepsilon})$ is proved. Furthermore, using $\calA'(1)=a$, we find that $\calA (\cos \varepsilon)$ 
is positive if $a$ is negative while $\calA (\cos \varepsilon)$ is negative if $a$ is positive. We thus have, provided 
that $\varepsilon$ is sufficiently small, $a\, \calA(z) \not \in \RR^+$ for any $z \in \Gamma_{\varepsilon,2}$. From 
now on, $\varepsilon$ is fixed and the latter properties hold.\\

\paragraph{\it Application of the residue theorem.}
We denote hereafter $\log_-$ the principal complex logarithm with the usual branch cut along $\RR_-$, and 
$\log_+$ the complex logarithm with a branch cut along $\RR_+$.

For any $z \in\Gamma_{\varepsilon,1}$, one has $\calA(z) \neq 0$ (thanks to assumption~\ref{as:circle}) and 
$\calA(z) \not \in \RR^{-,*}$ (for otherwise the stability region of \eqref{multipas} would contain $-\lambda \, \calA(z) 
\in \RR^{+,*}$, which can not hold by Lemma \ref{lm:multipas}). We can thus use $\log_-$ for computing the integral 
along $\Gamma_{\varepsilon,1}$, and we get
\[
\dfrac{1}{2\, i \, \pi} \, \int_{\Gamma_{\varepsilon,1}} \dfrac{\calA'(z)}{\calA(z)} \, {\rm d}z 
=\dfrac{1}{2\, i \, \pi} \, \Big( \log_- \calA ({\rm e}^{-i \, \varepsilon}) -\log_- \calA ({\rm e}^{i \, \varepsilon}) \Big)\, .
\]

The integral along $\Gamma_{\varepsilon,2}$ depends on the sign of $a$.
\begin{itemize}
   \item Suppose $a<0$. Then we know that for all $z \in \Gamma_{\varepsilon,2}$, $\calA(z)$ does not belong to 
   $\RR^-$. We can again use the $\log_-$ logarithm and derive
\[
\dfrac{1}{2\, i \, \pi} \, \int_{\Gamma_{\varepsilon,2}} \dfrac{\calA'(z)}{\calA(z)} \, {\rm d}z 
=\dfrac{1}{2\, i \, \pi} \, \Big( \log_- \calA ({\rm e}^{i \, \varepsilon}) -\log_- \calA ({\rm e}^{-i \, \varepsilon}) \Big)\, .
\]
   Summing the two contributions over $\Gamma_{\varepsilon,1}$ and $\Gamma_{\varepsilon,2}$ we finally obtain
\[
\dfrac{1}{2\, i \, \pi} \, \int_{\Gamma_{\varepsilon}} \dfrac{\calA'(z)}{\calA(z)} \, {\rm d}z = 0 \, ,
\]
   and $R=r$.

   \item Suppose now $a>0$. Then we know that for all $z \in \Gamma_{\varepsilon,2}$, $\calA(z)$ does not belong to 
   $\RR^+$. We use the $\log_+$ logarithm and derive
\[
\dfrac{1}{2\, i \, \pi} \, \int_{\Gamma_{\varepsilon,2}} \dfrac{\calA'(z)}{\calA(z)} \, {\rm d}z 
=\dfrac{1}{2\, i \, \pi} \, \Big( \log_+ \calA ({\rm e}^{i \, \varepsilon}) -\log_+ \calA ({\rm e}^{-i \, \varepsilon}) \Big)\, .
\]
   Summing the two contributions over $\Gamma_{\varepsilon,1}$ and $\Gamma_{\varepsilon,2}$, we obtain
\[
R-r =\dfrac{1}{2\, i \, \pi} \, \Big( \log_+ \calA ({\rm e}^{i \, \varepsilon}) -\log_- \calA ({\rm e}^{i \, \varepsilon}) \Big) 
-\dfrac{1}{2\, i \, \pi} \, \Big( \log_+ \calA ({\rm e}^{-i \, \varepsilon}) -\log_- \calA ({\rm e}^{-i \, \varepsilon}) \Big) \, .
\]
The difference $\log_{+}-\log_{-}$ equals $0$ on $\Omega_+ := \big\{ z \in \mathbb{C} \, ,\, \Im \, z>0 \big\}$ 
and equals $2\, i\, \pi$ on $\Omega_- := \big\{ z \in \mathbb{C} \, ,\, \Im \, z<0 \big\}$. To complete the proof, 
we recall that $\calA({\rm e}^{\pm i \, \varepsilon})$ belong to $\Omega_\pm$, and we thus get $R-r=-1$.
\end{itemize}
\end{proof}

\subsection{The leading boundary layer profile}

Using the flux function $F$, we can rewrite the boundary layer profile equations \eqref{eq:zerothBL}, 
\eqref{eq:zerothDirichlet}, and introduce the following definition.

\begin{definition}
\label{def:BL}
Being given a real number $u$, we call $(v_j)_{j\in\NN}$ a \emph{boundary layer profile} associated with $u$ 
a sequence that satisfies the following requirements:
\begin{enumerate}
\renewcommand{\labelenumi}{\it (\roman{enumi})}
\item $v_0 =\cdots =v_{r-1} =-u$,
\item $F(u+v_j,\ldots,u+v_{j+p+r-1}) =F(u,\dots,u)$, for all $j\geq 0$,
\item $\lim_{j \to \infty} v_j = 0$.
\end{enumerate}
\end{definition}

\noindent Let us comment some facts. The first point {\it (i)} above is related to the Dirichlet condition 
\eqref{eq:zerothDirichlet} with $u$ in place of $u^{\rm int}(0,t^n)$ (here the time variable is frozen). 
As a consequence of the linearity of the numerical flux~$F$, the above condition {\it (ii)} reads
$$
\forall \, j \ge 0 \, ,\quad \sum_{\ell = -r}^{p-1} f_\ell \, v_{j+\ell+r} =0 \, ,
$$
which is equivalent to
\begin{equation}
\label{eq:linearrecurrence}
\forall \, j \ge 0 \, ,\quad \sum_{\ell = -r}^p a_\ell \, v_{j+\ell+r} =0 \, ,
\end{equation}
if condition {\it (iii)} is satisfied. Boundary layer profiles are therefore the zero-limit solutions to the linear 
recurrence relation~\eqref{eq:linearrecurrence} for which the $r$ first terms of the sequence coincide.

\begin{definition}
The set of all the values $u$ such that a stable boundary layer associated to $u$ exists is denoted
\label{def:Cnum}
$$
\mathcal{C}_{\rm num} = \left\{ u\in\RR,\, \exists \, v\in\RR^\NN \textrm{ boundary layer profile associated with } 
u \right\} \, .
$$
\end{definition}

\noindent This definition is the same as in \cite{DuboisLeFloch,GisclonSerre,chainais-grenier}. The set 
$\mathcal{C}_{\rm num}$ encodes the so-called {\it residual boundary conditions} for \eqref{eq:transport} 
coming from the continuous limit $\Delta t \to 0$ in \eqref{eq:schemelinear}, \eqref{eq:homogDirichlet}. 
We are now ready to prove the following result that characterizes the boundary layer profiles for the 
scheme~\eqref{eq:schemelinear}.

\begin{proposition}
\label{prop:BL}
Under Assumptions~\ref{as:consistency}, \ref{as:consistencymultistep}, \ref{as:stability} and~\ref{as:circle}, 
there holds:
\begin{itemize}
   \item if $a>0$, then $\mathcal{C}_{\rm num}=\{ 0 \}$ and the unique boundary layer profile associated with 
   $0$ is the zero sequence ($v_j=0$ for all $j\geq 0$);

   \item if $a<0$, then $\mathcal{C}_{\rm num}=\RR$ and for any $u\in\RR$ there is a unique boundary layer 
   profile $(v_j)_{j\in\NN}$ associated with $u$, that decreases exponentially fast at infinity. We may write
\begin{equation}
\label{eq:boundmode}
v_j = u \, w_j \, ,\quad j\geq 0 \, ,
\end{equation}
where $(w_j)_{j\in\NN}$ denotes the boundary layer profile associated with $u=1$.
\end{itemize}
\end{proposition}

\begin{proof}
As explained above, our goal is to determine the zero-limit solutions to the recurrence \eqref{eq:linearrecurrence} 
that satisfy condition {\it (i)} in Definition \ref{def:BL}. We thus look for the (stable) roots to the polynomial equation
$$
\sum_{\ell = -r}^p a_\ell \, z^{\ell+r} =0 \, .
$$
Since this polynomial does not vanish at zero, its roots in $\DD$ coincide (with equal multiplicity) with the zeros 
of $\calA$ in $\DD \setminus \{ 0\}$. Lemma~\ref{lm:roots} precisely gives the number of such zeros.

The zero-limit solutions to the linear recurrence~\eqref{eq:linearrecurrence} are spanned by the sequences 
$Z^{(m)}$ ($m=1,\dots,r$ if $a<0$ and $m=1,\dots,r-1$ if $a>0$):
\begin{equation}
\label{eq:solhomog}
(j^\nu \, z_i^j)_{j \in \NN} \, ,\quad 0\leq \nu < \mu_i \, ,\ 1\leq i \leq q \, ,
\end{equation}
where $z_1,\dots,z_q$ denote the pairwise distinct zeros of $\calA$ in $\DD \setminus \{ 0\}$ and 
$\mu_1,\dots,\mu_q$ their corresponding multiplicity.

\begin{itemize}
   \item We first assume $a>0$. The subspace of zero-limit solutions to the linear recurrence~\eqref{eq:linearrecurrence} 
   has dimension $r-1$. Let $u \in \RR$. We are looking for a sequence $v = \sum_{m=1}^{r-1} \omega_m \, Z^{(m)}$ 
   such that $v_0 =\cdots =v_{r-1}=-u$, which is equivalent to
\[
\begin{pmatrix}
Z^{(1)}_0 & \dots & Z^{(r-1)}_0 & 1 \\
\vdots &  & \vdots & \vdots \\
Z^{(1)}_{r-1} & \dots & Z^{(r-1)}_{r-1} & 1
\end{pmatrix}
\begin{pmatrix}
\omega_1 \\ \vdots \\ \omega_{r-1} \\ u
\end{pmatrix}
= 0 \, .
\]
   The involved matrix in $\mathcal{M}_{r,r}(\mathbb{C})$ is invertible and therefore $u=0$, $v=0$.

   \item We now assume $a<0$. The subspace of zero-limit solutions to the linear recurrence~\eqref{eq:linearrecurrence} 
   has dimension $r$. Let $u \in \RR$. We are looking for a sequence $v = \sum_{m=1}^r \omega_m \, Z^{(m)}$ 
   such that $v_0 =\cdots =v_{r-1}=-u$, which is equivalent to
\[
\begin{pmatrix}
Z^{(1)}_0 & \dots & Z^{(r)}_0 \\
\vdots &  & \vdots \\
Z^{(1)}_{r-1} & \dots & Z^{(r)}_{r-1}
\end{pmatrix}
\begin{pmatrix}
\omega_1 \\ \vdots \\ \omega_r 
\end{pmatrix}
=-u
\begin{pmatrix}
1 \\ \vdots \\ 1
\end{pmatrix} \, .
\]
   The involved matrix of $\mathcal{M}_{r,r}(\mathbb{C})$ is invertible and thus, for each given $u\in\RR$ there 
   is a unique solution $(\omega_1,\ldots,\omega_r)\in\mathbb{C}^r$ which determines the boundary layer profile 
   associated with $u$. By linearity, this profile takes the form \eqref{eq:boundmode} and it is exponentially 
   decreasing.
\end{itemize}
\end{proof}

\subsection{The first boundary layer corrector}

Our goal in this subsection is to construct a solution to the first boundary layer corrector equations \eqref{eq:firstBL}, 
\eqref{eq:firstDirichlet}. In what follows, the function $u^{\rm bl,0}$ will be a boundary layer profile associated with 
some discretized trace of the exact solution to \eqref{eq:transport}. In the case $a>0$, there is no boundary layer 
profile but zero and the solution to \eqref{eq:firstBL}, \eqref{eq:firstDirichlet} is also zero. In the case $a<0$, the 
space of boundary layer profiles is spanned by the sequence $(w_j)_{j \in \NN}$, and it is therefore sufficient to 
construct a sequence that satisfies
\begin{align}
\forall \, j \ge r \, ,\quad & \sum_{\ell=-r}^p a_\ell \, \widetilde{w}_{j+\ell} +w_j =0 \, ,\label{defwjtilde}\\
& \widetilde{w}_0 =\cdots =\widetilde{w}_{r-1} =0 \, ,\quad \lim_{j\to \infty} \widetilde{w}_j = 0 \, .\notag
\end{align}

\begin{lemma}
\label{lm:wjtilde}
Under the assumptions of Proposition \ref{prop:BL}, in the case $a<0$, there exists a unique solution 
$(\widetilde{w}_j)_{j\in \NN}$ to \eqref{defwjtilde} and this solution decays exponentially fast at infinity.
\end{lemma}

\begin{proof}
Uniqueness easily follows from the linearity of \eqref{defwjtilde} and Proposition \ref{prop:BL}. As far as existence 
is concerned, we keep the notation of Proposition \ref{prop:BL} and decompose the sequence $(w_j)_{j\in \NN}$ 
as
$$
w=\sum_{m=1}^r \omega_m \, Z^{(m)} \, ,
$$
where the $Z^{(m)}$'s are given by \eqref{eq:solhomog}. Due to the linearity of \eqref{defwjtilde}, we first construct 
a zero-limit solution to the recurrence
$$
\forall \, j \ge r \, ,\quad \sum_{\ell=-r}^p a_\ell \, W_{j+\ell}^{(m)} +Z^{(m)}_j =0 \, ,\quad 
Z^{(m)}_j =j^\nu \, z_i^j \, ,
$$
which is done by choosing $W^{(m)}$ of the form
$$
W_j^{(m)} =\sum_{\mu=0}^{\mu_i-1} \varsigma_\mu \, j^{\mu_i +\mu} \, z_i^j \, ,
$$
and by identifying the coefficients $\varsigma_0,\dots,\varsigma_{\mu_i-1}$ (this procedure gives an invertible 
upper triangular system). Summing finitely many such sequences $W^{(m)}$, we get a sequence $W$ that 
decays exponentially at infinity and that is a solution to the recurrence relation
$$
\forall \, j \ge r \, ,\quad \sum_{\ell=-r}^p a_\ell \, W_{j+\ell} +w_j =0 \, .
$$
The sequence $(\widetilde{w}_j)$ is obtained by correcting the initial conditions for $(W_j)$, that is by choosing
$$
\widetilde{w} :=W +\sum_{m=1}^r \varpi_m \, Z^{(m)} \, ,
$$
with
\[
\begin{pmatrix}
Z^{(1)}_0 & \dots & Z^{(r)}_0 \\
\vdots &  & \vdots \\
Z^{(1)}_{r-1} & \dots & Z^{(r)}_{r-1}
\end{pmatrix}
\begin{pmatrix}
\varpi_1 \\ \vdots \\ \varpi_r 
\end{pmatrix} = -\begin{pmatrix}
W_0 \\ \vdots \\ W_{r-1} \end{pmatrix} \, .
\]
\end{proof}

\subsection{The approximate solution}
\label{defuapp}

Let us recall that the solution to \eqref{eq:transport}, supplemented with the homogeneous Dirichlet 
condition \eqref{eq:Dirichlet} in the case $a>0$, is given by
\begin{equation}
\label{eq:uint}
u^{\rm ex}(x,t) =u_0(x-a\, t),\quad x\geq 0,\quad t\geq 0 \, ,
\end{equation}
where the initial condition $u_0$ has been extended by $0$ to $\RR^-$ in the case $a>0$. This suggests 
defining the interior numerical solution as
\begin{equation*}
u^{\rm int}_{j,n} := \dfrac{1}{\Delta x} \, \int_{x_j}^{x_{j+1}} u_0(x-a\, t^n)\, {\rm d}x \, ,\quad j\geq 0 \, ,\quad 
n\geq 0 \, .
\end{equation*}
In particular, \eqref{eq:initialdata} gives
\begin{equation*}
\forall \, n=0,\dots,k-1 \, ,\quad \forall \, j \ge 0 \, ,\quad u^{\rm int}_{j,n}=u_j^n \, .
\end{equation*}

In the case $a>0$, there is no boundary layer and we define the approximate solution $u^{\rm app}$ to 
\eqref{eq:schemelinear}, \eqref{eq:homogDirichlet} as
\begin{equation*}
u^{\rm app}_{j,n} := u^{\rm int}_{j,n} \, ,\quad j\geq 0 \, ,\quad n\geq 0 \, .
\end{equation*}

In the case $a<0$, there exists a one-dimensional space of boundary layer profiles and we can also construct 
boundary layer correctors. In view of \eqref{eq:zerothDirichlet}, we first need to approximate the trace of the 
exact solution $u^{\rm ex}$ and therefore set
\begin{equation*}
\forall \, n \ge 0 \, ,\quad u^{\rm tr}_n := \dfrac{1}{\Delta t} \, \int_{t^n}^{t^{n+1}} u_0(-a \, t) \, {\rm d}t \, .
\end{equation*}
We now define the leading order boundary layer profile $u^{\rm bl,0}$ and first order boundary layer corrector 
$u^{\rm bl,1}$ as follows:
$$
u^{\rm bl,0}_{j,n} :=\begin{cases}
0 \, ,& j\ge 0 \, ,\quad n=0,\dots,k-1 \, ,\\
u^{\rm tr}_n \, w_j  \, ,& j\ge 0 \, ,\quad n \ge k \, ,
\end{cases}
$$
$$
u^{\rm bl,1}_{j,n} :=\begin{cases}
0 \, ,& j\ge 0 \, ,\quad n=0,\dots,k-1 \, ,\\
\left( \Delta t \, \sum_{\sigma=0}^{k-1} \beta_\sigma \right)^{-1} \, 
\left( \sum_{\sigma=0}^k \alpha_\sigma \, u^{\rm tr}_{n+\sigma} \right) \, \widetilde{w}_j  \, ,& 
j\ge 0 \, ,\quad n \ge k \, .
\end{cases}
$$
The approximate solution $u^{\rm app}$ to \eqref{eq:schemelinear}, \eqref{eq:homogDirichlet} is then defined by:
\begin{equation}
\label{eq:1ordApprox}
u^{\rm app}_{j,n} := u^{\rm int}_{j,n} +u^{\rm bl,0}_{j,n} +\Delta x \, u^{\rm bl,1}_{j,n} \, ,\quad j\geq 0 \, ,\quad 
n\geq 0 \, .
\end{equation}
Thanks to our choice for the initial data, we again have:
\begin{equation*}
u^{\rm app}_{j,n} = u_j^n \, ,\quad j\geq 0 \, ,\quad n=0,\dots,k-1 \, .
\end{equation*}

\section{Proof of the main result}
\label{sect3}

The error analysis uses the expression of the approximate solution $(u^{\rm app}_{j,n})_{j\geq 0,n\geq 0}$ 
introduced in subsection~\ref{defuapp}. We thus focus on the interior consistency error that is defined by:
\begin{equation*}
\varepsilon_{j,n+k} := \dfrac{1}{\Delta t} \, \left( \sum_{\sigma=0}^k \alpha_\sigma \, u^{\rm app}_{j,n+\sigma} 
+\lambda \, \sum_{\sigma=0}^{k-1} \beta_\sigma \, \sum_{\ell=-r}^p a_\ell \, u^{\rm app}_{j+\ell,n+\sigma} \right) \, ,
\end{equation*}
with  $j\geq r$ and $n\geq 0$, and on the boundary errors:
\begin{equation*}
\eta_{j,n} := u^{\rm app}_{j,n},\quad 0 \le j\leq r-1 \, ,\quad n\geq k \, .
\end{equation*}
We recall that the approximate solution $u^{\rm app}$ has the same initial data as the exact numerical solution 
(whatever the sign of $a$):
\begin{equation*}
u^{\rm app}_{j,n}=u_j^n \, ,\quad j\geq 0 \, ,\quad n=0,\dots,k-1 \, .
\end{equation*}
Consequently, the error:
\begin{equation*}
e_{j,n} :=u^{\rm app}_{j,n}-u_j^n \, ,\quad j\geq 0 \, ,\quad n\geq 0 \, ,
\end{equation*}
is a solution to the following numerical scheme with presumably small forcing terms and zero initial data:
\begin{equation}
\label{eq:error}
\begin{cases}
\sum_{\sigma=0}^k \alpha_\sigma \, e_{j,n+\sigma} 
+\lambda \, \sum_{\sigma=0}^{k-1} \beta_\sigma \, \sum_{\ell=-r}^p a_\ell \, e_{j+\ell,n+\sigma} 
=\Delta t \, \varepsilon_{j,n+k} \, ,& j\geq r \, ,\quad n\geq 0 \, ,\\
e_{j,n} = \eta_{j,n} \, ,& 0 \le j\leq r-1 \, ,\quad n\geq k \, ,\\
e_{j,0} = \cdots = e_{j,k-1}= 0 \, ,& j\geq 0 \, .
\end{cases}
\end{equation}
The aim of the following two subsections is to quantify the smallness of the source terms in \eqref{eq:error} 
in order to apply the stability estimate of \cite{goldberg-tadmor}. The smallness of the source terms will yield, 
up to losing some powers of $\Delta t$, a semigroup estimate for $(e_{j,n})_{j\geq 0, n\geq 0}$ which will 
eventually give the semigroup estimate for the numerical solution $(u_j^n)_{j\geq 0, n\geq 0}$.

\subsection{The case of an incoming velocity}

We assume here $a>0$ so that no boundary layer arises in the solution to \eqref{eq:schemelinear}, 
\eqref{eq:homogDirichlet} ($\mathcal{C}_{\rm num}=\{ 0 \}$). The approximate solution merely reads:
\begin{equation*}
u^{\rm app}_{j,n} =\dfrac{1}{\Delta x} \, \displaystyle \int_{x_j}^{x_{j+1}} u_0(y -a\, t^n)\, {\rm d}y \, ,\quad 
j \geq 0 \, ,\quad n\geq 0 \, ,
\end{equation*}
where we recall that $u_0$ has been extended by zero to $\RR^-$. From the flatness conditions 
$u_0(0)=u_0'(0)=0$, we have $u_0 \in H^2(\RR)$. The errors in \eqref{eq:error} satisfy the following 
bounds.

\begin{proposition}
\label{prop:errortermsPOS}
Let us assume $a>0$. 
Under the assumptions of Theorem~\ref{thm:stab1step} and in the CFL regime~\eqref{eq:CFLnumber}, 
there exists a constant $C>0$ that is independent of $u_0$ and $\Delta t\in(0,1]$ such that
\begin{gather}
\label{eq:internalL2L2POS}
\sup_{n \ge k} \, \sum_{j \geq r} \Delta x \, |\varepsilon_{j,n}|^2 \leq C \, \Delta t^2 \, \|u_0''\|_{L^2(\RR^+)}^2 \, ,\\
\sum_{n \ge k} \, \sum_{j=0}^{r-1} \Delta t \, |\eta_{j,n}|^2 \leq C \, \Delta t^2 \, \|u_0'\|_{L^2(\RR^+)}^2 
\, .\label{eq:boundarylL2L2POS}
\end{gather}
\end{proposition}

\begin{proof}
Let us first consider the boundary error terms $(\eta_{j,n})$. Since $u_0$ vanishes on $\RR^-$, there holds 
$\eta_{j,n}=0$ if $n \ge r/(a\, \lambda)$. The sum in \eqref{eq:boundarylL2L2POS} therefore reduces to finitely 
many terms (and the number of such terms is independent of $\Delta t$). We consider some space index 
$j \in \{ 0,\dots,r-1\}$ and some time index $k \le n<r/(a\, \lambda)$, and write
$$
\eta_{j,n}=\dfrac{1}{\Delta x} \, \int_{x_j}^{x_{j+1}} u_0(x-a\, t^n) \, {\rm d}x 
=\dfrac{1}{\Delta x} \, \int_{x_j}^{x_{j+1}} \! \! \int_0^{x-a\, t^n} u_0'(y) \, {\rm d}y \, {\rm d}x \, .
$$
We then apply the Cauchy-Schwarz inequality and get
$$
|\eta_{j,n}|^2 \le C\, \int_{x_j}^{x_{j+1}} \left| \int_0^{x-a\, t^n} u_0'(y)^2 \, {\rm d}y \right| \, {\rm d}x 
\le C\, \Delta t \, \| u_0' \|_{L^2(\RR^+)}^2 \, .
$$
Summing the finitely many nonzero error terms, we get \eqref{eq:boundarylL2L2POS}.

We now deal with the consistency error in the interior domain. Using the consistency 
assumptions~\ref{as:consistency} and \ref{as:consistencymultistep}, we have:
\[
\begin{aligned}
\Delta t \, \varepsilon_{j,n+k} &= \dfrac{1}{\Delta x} \, \sum_{\sigma=0}^k \alpha_\sigma \, 
\int_{x_j}^{x_{j+1}} u_0(x-a\, t^{n+\sigma}) -u_0(x-a\, t^n) \, {\rm d}x \\
&+\dfrac{\lambda}{\Delta x} \, \sum_{\sigma=0}^{k-1} \beta_\sigma \, \sum_{\ell=-r}^p a_\ell \, 
\left( \int_{x_{j+\ell}}^{x_{j+\ell+1}} u_0(x-a\, t^{n+\sigma})  \, {\rm d}x 
-\int_{x_j}^{x_{j+1}} u_0(x-a\, t^{n+\sigma}) \, {\rm d}x \right) \\
&= -\dfrac{1}{\Delta x} \, \sum_{\sigma=0}^k \alpha_\sigma \, 
\int_{x_j}^{x_{j+1}} \! \! \int_{-a\, \sigma \, \Delta t}^0 u_0'(x+y-a\, t^n) \, {\rm d}y \, {\rm d}x \\
&+\dfrac{\lambda}{\Delta x} \, \sum_{\sigma=0}^{k-1} \beta_\sigma \, \sum_{\ell=-r}^p a_\ell \, 
\int_{x_j}^{x_{j+1}} \! \! \int_0^{\ell \, \Delta x} u_0'(x+y-a\, t^{n+\sigma}) \, {\rm d}y \, {\rm d}x \\
&= -\dfrac{\lambda}{\Delta x} \, \sum_{\sigma=0}^k \sigma \, \alpha_\sigma \, a \, 
\int_{x_j}^{x_{j+1}} \! \! \int_0^{\Delta x} u_0'(x-a\, \sigma \, \lambda \, y-a\, t^n) \, {\rm d}y \, {\rm d}x \\
&+\dfrac{\lambda}{\Delta x} \, \sum_{\sigma=0}^{k-1} \beta_\sigma \, \sum_{\ell=-r}^p \ell \, a_\ell \, 
\int_{x_j}^{x_{j+1}} \! \! \int_0^{\Delta x} u_0'(x+\ell \, y-a\, t^{n+\sigma}) \, {\rm d}y \, {\rm d}x \, .
\end{aligned}
\]
Using the consistency assumptions~\ref{as:consistency} and \ref{as:consistencymultistep} again, we can 
add the zero quantity
$$
\dfrac{\lambda}{\Delta x} \, \left( \sum_{\sigma=0}^k \sigma \, \alpha_\sigma \, a 
-\sum_{\sigma=0}^{k-1} \beta_\sigma \, \sum_{\ell=-r}^p \ell \, a_\ell \right) \, 
\int_{x_j}^{x_{j+1}} \! \! \int_0^{\Delta x} u_0'(x-a\, t^n) \, {\rm d}y \, {\rm d}x \, ,
$$
and get
\begin{align}
\Delta t \, \varepsilon_{j,n+k} &= \dfrac{\lambda}{\Delta x} \, \sum_{\sigma=0}^k \sigma \, \alpha_\sigma \, a \, 
\int_{x_j}^{x_{j+1}} \! \! \int_0^{\Delta x} \! \! \int_{-a\, \sigma \, \lambda \, y}^0 u_0''(x+x'-a\, t^n) \, 
{\rm d}x' \, {\rm d}y \, {\rm d}x \notag\\
&+\dfrac{\lambda}{\Delta x} \, \sum_{\sigma=0}^{k-1} \beta_\sigma \, \sum_{\ell=-r}^p \ell \, a_\ell \, 
\int_{x_j}^{x_{j+1}} \! \! \int_0^{\Delta x} \! \! \int_0^{\ell \, y-a\, \sigma \, \Delta t} u_0''(x+x'-a\, t^n) \, 
{\rm d}x' \, {\rm d}y \, {\rm d}x \, .\label{decomperror1}
\end{align}

We now apply successive Cauchy-Schwarz inequalities. In the CFL regime~\eqref{eq:CFLnumber}, we get 
for instance
$$
\begin{aligned}
\left| \int_{x_j}^{x_{j+1}} \! \! \int_0^{\Delta x} \! \! \int_{-a\, \sigma \, \lambda \, y}^0 u_0''(x+x'-a\, t^n) \, 
{\rm d}x' \, {\rm d}y \, {\rm d}x \right|^2 &\le C\, \Delta t^3 \, 
\int_{x_j}^{x_{j+1}} \! \! \int_0^{\Delta x} \! \! \int_{-a\, \sigma \, \lambda \, y}^0 u_0''(x+x'-a\, t^n)^2 \, 
{\rm d}x' \, {\rm d}y \, {\rm d}x \\
&\le C \, \Delta t^4 \, \int_{x_j}^{x_{j+1}} \! \! \int_{-a\, \sigma \, \Delta t}^0 u_0''(x+x'-a\, t^n)^2 \, {\rm d}x' \, {\rm d}x \\
&\le C \, \Delta t^5 \, \int_{x_j-a\, k \, \Delta t}^{x_{j+1}} u_0''(x-a\, t^n)^2 \, {\rm d}x \, .
\end{aligned}
$$
The other error term in \eqref{decomperror1} is estimated similarly, and in the end, we can show that there 
exists a fixed integer $j_0>0$ (that only depends on the CFL number $\lambda$, $a$ and $k$) such that
$$
\forall \, j \ge r \, ,\quad \forall \, n \in \NN \, ,\quad |\varepsilon_{j,n+k}|^2 \le C \, \Delta t \, 
\int_{x_{j-j_0}}^{x_{j+p+1}} u_0''(x-a\, t^n)^2 \, {\rm d}x \, .
$$
The estimate \eqref{eq:internalL2L2POS} follows immediately.
\end{proof}

\subsection{The case of an outgoing velocity}

From now on, we consider the case of an outgoing velocity $a<0$ for which non-trivial boundary layers 
appear in the solution to the numerical scheme \eqref{eq:schemelinear}, \eqref{eq:homogDirichlet} 
($\mathcal{C}_{\rm num}=\RR$). The following Proposition provides error bounds for the source terms 
in the numerical scheme \eqref{eq:error}.

\begin{proposition}
\label{prop:errorterms}
Under the assumptions of Theorem~\ref{thm:stab1step} and in the CFL regime~\eqref{eq:CFLnumber}, 
there exists a constant $C>0$ that is independent of $u_0$ and $\Delta t\in(0,1]$ such that
\begin{gather}
\label{eq:internalL2L2}
\sup_{n \geq 2\, k} \, \sum_{j\geq r} \Delta x \, |\varepsilon_{j,n}|^2 \leq C \, \Delta t^2 \, 
\|u_0''\|_{L^2(\RR^+)}^2 \, ,\\
\label{eq:internalL2L2bis}
\sup_{k \le n \le 2\, k-1} \, \sum_{j\geq r} \Delta x \, |\varepsilon_{j,n}|^2 \leq C \, \Delta t \, 
\|u_0'\|_{H^1(\RR^+)}^2 \, ,\\
\sum_{n\geq k} \, \sum_{j=0}^{r-1} \Delta t \, |\eta_{j,n}|^2 \leq C \, \Delta t^2 \, \|u_0'\|_{L^2(\RR^+)}^2 
\, .\label{eq:boundaryL2L2}
\end{gather}
\end{proposition}

\begin{proof}
We first prove \eqref{eq:boundaryL2L2} and then deal with \eqref{eq:internalL2L2} and \eqref{eq:internalL2L2bis}.

\paragraph{\it Errors at the boundary.}
We start with the proof of the estimate \eqref{eq:boundaryL2L2}. From the definition \eqref{eq:1ordApprox}, 
we obtain (recall $n \ge k$ and $j=0,\dots,r-1$):
\[
\eta_{j,n} = u^{\rm app}_{j,n} = \dfrac{1}{\Delta x} \, \int_{x_j}^{x_{j+1}} u_0(x+|a| \, t^n) \, {\rm d}x 
-\dfrac{1}{\Delta t} \, \int_{t^n}^{t^{n+1}} u_0(|a| \, t)\, {\rm d}t.
\]
With the notation \eqref{eq:CFLnumber}, the error $\eta_{j,n}$ can be written as
\[
\begin{aligned}
\eta_{j,n} &= \dfrac{\tau}{\Delta t} \, \int_0^{\Delta t/\tau} u_0(x_j+|a| \, t^n+|a|\, s)-u_0(|a| \, t^n) \, {\rm d}s 
-\dfrac{1}{\Delta t} \, \int_0^{\Delta t} u_0(|a| \, t^n+|a|\, s) -u_0(|a| \, t^n) \, {\rm d}s \\
&= \dfrac{\tau}{\Delta t} \, \int_0^{\Delta t/\tau} \int_0^{x_j+|a|\, s} u_0'(|a| \, t^n+y) \, {\rm d}y  \, {\rm d}s 
-\dfrac{1}{\Delta t} \, \int_0^{\Delta t} \int_0^{|a|\, s} u_0'(|a| \, t^n+y) \, {\rm d}y  \, {\rm d}s \, .
\end{aligned}
\]
Each term in $\eta_{j,n}$ is estimated by applying the Cauchy-Schwarz inequality. For instance, we have
$$
\left| \dfrac{1}{\Delta t} \, \int_0^{\Delta t} \int_0^{|a|\, s} u_0'(|a| \, t^n+y) \, {\rm d}y  \, {\rm d}s \right|^2 
\le C\, \Delta t \, \int_0^{|a|Ê\, \Delta t} u_0'(|a| \, t^n+y)^2 \, {\rm d}y \, ,
$$
and similarly, we have
$$
\left| \dfrac{\tau}{\Delta t} \, \int_0^{\Delta t/\tau} \int_0^{x_j+|a|\, s} u_0'(|a| \, t^n+y) \, {\rm d}y  \, {\rm d}s \right|^2 
\le C\, \Delta t \, \int_0^{rÊ\, \Delta x} u_0'(|a| \, t^n+y)^2 \, {\rm d}y \, .
$$
Summing over the $n$'s and the finitely many $j$'s, we derive the bound \eqref{eq:boundaryL2L2}.

\paragraph{\it Errors in the interior.}
We decompose the consistency error $\varepsilon_{j,n}$ in \eqref{eq:error} as
$$
\varepsilon_{j,n} =\varepsilon^{\rm int}_{j,n} +\varepsilon^{\rm bl}_{j,n} \, ,
$$
with self-explanatory notation. The estimate of the {\it interior} consistency error $\varepsilon^{\rm int}_{j,n}$ 
follows from the exact same arguments as we used in the case of an incoming velocity. The only difference 
is that, because of the sign of $a$, we do not need to extend $u_0$ by zero to $\RR^-$ and no assumption 
on the behavior of $u_0$ at $0$ is needed to derive the estimate
\begin{equation}
\label{estimerrorint}
\sup_{n \geq k} \, \sum_{j\geq r} \Delta x \, |\varepsilon^{\rm int}_{j,n}|^2 \leq C \, \Delta t^2 \, 
\| u_0'' \|_{L^2(\RR^+)}^2 \, .
\end{equation}
We now focus on the new consistency error that comes from the boundary layer terms in $u^{\rm app}$:
\begin{align*}
\varepsilon^{\rm bl}_{j,n+k} &=\dfrac{1}{\Delta t} \, \left( \sum_{\sigma=0}^k \alpha_\sigma \, u^{\rm bl,0}_{j,n+\sigma} 
+\lambda \, \sum_{\sigma=0}^{k-1} \beta_\sigma \, \sum_{\ell=-r}^p a_\ell \, u^{\rm bl,0}_{j+\ell,n+\sigma} \right) 
+\dfrac{1}{\lambda} \, \sum_{\sigma=0}^k \alpha_\sigma \, u^{\rm bl,1}_{j,n+\sigma} 
+\sum_{\sigma=0}^{k-1} \beta_\sigma \, \sum_{\ell=-r}^p a_\ell \, u^{\rm bl,1}_{j+\ell,n+\sigma} \\
&=\dfrac{1}{\Delta t} \, \sum_{\sigma=0}^k \alpha_\sigma \, u^{\rm bl,0}_{j,n+\sigma} 
+\dfrac{1}{\lambda} \, \sum_{\sigma=0}^k \alpha_\sigma \, u^{\rm bl,1}_{j,n+\sigma} 
+\sum_{\sigma=0}^{k-1} \beta_\sigma \, \sum_{\ell=-r}^p a_\ell \, u^{\rm bl,1}_{j+\ell,n+\sigma} \, .
\end{align*}
In the case $n \ge k$, we use the definition of the boundary layer profile and corrector $u^{\rm bl,0}$, 
$u^{\rm bl,1}$ to simplify the latter expression and get\footnote{Here we use the consistency assumption 
\ref{as:consistencymultistep}.}
\begin{equation*}
\varepsilon^{\rm bl}_{j,n+k} =\dfrac{1}{\lambda} \, \sum_{\sigma=0}^k \alpha_\sigma \, 
(u^{\rm bl,1}_{j,n+\sigma} -u^{\rm bl,1}_{j,n}) 
+\sum_{\sigma=0}^{k-1} \beta_\sigma \, \sum_{\ell=-r}^p a_\ell \, (u^{\rm bl,1}_{j+\ell,n+\sigma} 
-u^{\rm bl,1}_{j+\ell,n}) \, .
\end{equation*}
The first boundary layer corrector is given in subsection \ref{defuapp}. In particular, the error 
$\varepsilon^{\rm bl}_{j,n+k}$ can be decomposed as a linear combination of sequences of the form
$$
\dfrac{\widetilde{w}_{j+\ell}}{\Delta t} \, \sum_{\sigma'=0}^k \alpha_{\sigma'} \, (u_{n+\sigma+\sigma'}^{\rm tr} 
-u_{n+\sigma'}^{\rm tr}) \, ,
$$
with $\sigma =0,\dots,k$ and $\ell=-r,\dots,p$. Since the sequence $(\widetilde{w}_j)$ is exponentially decreasing, 
we have
$$
\begin{aligned}
\sup_{n \ge k} \, \sum_{j \ge r} \Delta x \, |\varepsilon^{\rm bl}_{j,n+k}|^2 &\le C \, \Delta x \, 
\sup_{n \ge k} \, \sum_{\sigma=0}^k \dfrac{1}{\Delta t^2} \, \left| \sum_{\sigma'=0}^k 
\alpha_{\sigma'} \, (u_{n+\sigma+\sigma'}^{\rm tr} -u_{n+\sigma'}^{\rm tr}) \right|^2 \\
&\le \dfrac{C}{\Delta t} \, \sup_{n \ge k} \, \sum_{\sigma=0}^k \left| \sum_{\sigma'=0}^k \alpha_{\sigma'} \, 
(u_{n+\sigma+\sigma'}^{\rm tr} -u_{n+\sigma'}^{\rm tr} -u_{n+\sigma}^{\rm tr} +u_n^{\rm tr}) \right|^2 \, .
\end{aligned}
$$
We compute
$$
u_{n+\sigma+\sigma'}^{\rm tr} -u_{n+\sigma'}^{\rm tr} -u_{n+\sigma}^{\rm tr} +u_n^{\rm tr} 
=\dfrac{a^2}{\Delta t} \, \int_0^{\Delta t} \! \! \int_0^{\sigma \, \Delta t} \! \! \int_0^{\sigma' \, \Delta t} 
u_0''(|a| \, t^n +|a| \, s_1 +|a| \, s_2 +|a| \, s_3) \, {\rm d}s_3 \, {\rm d}s_2 \, {\rm d}s_1 \, ,
$$
and the Cauchy-Schwarz inequality yields
$$
|u_{n+\sigma+\sigma'}^{\rm tr} -u_{n+\sigma'}^{\rm tr} -u_{n+\sigma}^{\rm tr} +u_n^{\rm tr}|^2 
\le C \, \Delta t^3 \, \int_0^{(2\, k+1) \, \Delta t} u_0''(|a| \, t^n +|a| \, s)^2 \, {\rm d}s \, .
$$
We have thus derived the estimate
$$
\sup_{n \ge k} \, \sum_{j \ge r} \Delta x \, |\varepsilon^{\rm bl}_{j,n+k}|^2 \le C \, \Delta t^2 \, 
\| u_0'' \|_{L^2(\RR^+)}^2 \, .
$$
Together with \eqref{estimerrorint}, this already proves \eqref{eq:internalL2L2}.

We turn to the proof of \eqref{eq:internalL2L2bis}. It only remains to estimate the $\ell^2_j$ norm 
of $(\varepsilon^{\rm bl}_{j,k}),\dots,(\varepsilon^{\rm bl}_{j,2\, k-1})$ because the interior errors 
$(\varepsilon^{\rm int}_{j,k}),\dots,(\varepsilon^{\rm int}_{j,2\, k-1})$ already satisfy \eqref{estimerrorint}, 
which is not larger than the right hand side in \eqref{eq:internalL2L2bis}. Let us explain how we derive the 
estimate for $(\varepsilon^{\rm bl}_{j,k})$. The remaining terms are similar. The error $\varepsilon^{\rm bl}_{j,k}$ 
reads
$$
\varepsilon^{\rm bl}_{j,k} =\dfrac{1}{\Delta t} \, u_{j,k}^{\rm bl,0} +\dfrac{1}{\lambda} \, u_{j,k}^{\rm bl,1} 
=\dfrac{w_j}{\Delta t} \, u_k^{\rm tr} +\dfrac{\widetilde{w_j}}{\lambda} \, \left( \Delta t \, 
\sum_{\sigma=0}^{k-1} \beta_\sigma \right)^{-1} \, 
\sum_{\sigma=0}^k \alpha_\sigma \, u^{\rm tr}_{k+\sigma} \, ,
$$
so we have
$$
\sum_{j \ge r} \Delta x \, |\varepsilon^{\rm bl}_{j,k}|^2 \le \dfrac{C}{\Delta t} \, \sum_{\sigma=0}^k 
|u^{\rm tr}_{k+\sigma}|^2 \, .
$$
We now use the assumption $u_0(0)=0$ of Theorem~\ref{thm:stab1step} and get
$$
u^{\rm tr}_{k+\sigma} =\dfrac{1}{\Delta t} \, \int_0^{\Delta t} \! \! \int_0^{|a|\, (t^{k+\sigma} +s)} 
u_0'(y) \, {\rm d}y \, {\rm d}s \, .
$$
The Cauchy-Schwarz inequality then gives
$$
|u^{\rm tr}_{k+\sigma}|^2 \le C\, \Delta t \, \int_0^{|a|\, t^{k+\sigma+1}} u_0'(y)^2 \, {\rm d}y 
\le C\, \Delta t^2 \, \| u_0' \|_{L^\infty (\RR^+)}^2 \le C\, \Delta t^2 \, \| u_0' \|_{H^1 (\RR^+)}^2 \, .
$$
We thus get \eqref{eq:internalL2L2bis} for the sequence $(\varepsilon_{j,k})$ and the remaining terms 
$(\varepsilon^{\rm bl}_{j,k+1}),\dots,(\varepsilon^{\rm bl}_{j,2\, k-1})$ are dealt with in the same (rather 
crude) way.
\end{proof}

\noindent Propositions \ref{prop:errortermsPOS} and \ref{prop:errorterms} imply the following result which uses 
GKS type norms.

\begin{proposition}
\label{prop:errortermsgamm}
Under the assumptions of Theorem~\ref{thm:stab1step} and in the CFL regime~\eqref{eq:CFLnumber}, 
there exists a constant $C>0$ that is independent of $u_0$ and $\Delta t \in (0,1]$, such that for all 
$\gamma>0$ there holds
\begin{gather}
\label{eq:internalerrorL2L2gam}
\sum_{n \geq k} \, \sum_{j \geq r} \Delta t \, \Delta x \, {\rm e}^{-2\, n\, \gamma \, \Delta t} \, |\varepsilon_{j,n}|^2 
\leq C \, \left( 1+\dfrac{1}{\gamma} \right) \, \Delta t^2 \, \|u_0\|_{H^2(\RR^+)}^2 \, ,\\
\label{eq:boundaryerrorL2L2gam}
\sum_{n\geq k} \, \sum_{j=0}^{r-1} \Delta t \, {\rm e}^{-2\, n\, \gamma \,\Delta t} \, |\eta_{j,n}|^2 \leq C \, \Delta t^2 
\, \|u_0\|_{H^1(\RR^+)}^2 \, .
\end{gather}
\end{proposition}

\begin{proof}[Proof of Proposition~\ref{prop:errortermsgamm}]
The proof of \eqref{eq:boundaryerrorL2L2gam} is immediate and follows from either \eqref{eq:boundarylL2L2POS} 
or \eqref{eq:boundaryL2L2} by using $\gamma>0$ (so that the exponential factors in \eqref{eq:boundaryerrorL2L2gam} 
are not larger than $1$).

The proof of \eqref{eq:internalerrorL2L2gam} follows from either \eqref{eq:internalL2L2POS} or 
\eqref{eq:internalL2L2}-\eqref{eq:internalL2L2bis}. In the incoming case ($a>0$), we use \eqref{eq:internalL2L2POS} 
and get
\begin{multline*}
\sum_{n \geq k} \, \sum_{j \geq r} \Delta t \, \Delta x \, {\rm e}^{-2\, n\, \gamma \, \Delta t} \, |\varepsilon_{j,n}|^2 
\le C \, \Delta t^3 \, \|u_0\|_{H^2(\RR^+)}^2 \, \sum_{n \geq k} \, {\rm e}^{-2\, n\, \gamma \, \Delta t} \\
\le \dfrac{C}{{\rm e}^{2\, \gamma \, \Delta t}-1} \, \Delta t^3 \, \|u_0\|_{H^2(\RR^+)}^2 
\le \dfrac{C}{\gamma} \, \Delta t^2 \, \|u_0\|_{H^2(\RR^+)}^2 \, ,
\end{multline*}
which is even better than \eqref{eq:internalerrorL2L2gam}. In the outgoing case, we use 
\eqref{eq:internalL2L2}-\eqref{eq:internalL2L2bis} and get
\begin{multline*}
\sum_{n \geq k} \, \sum_{j \geq r} \Delta t \, \Delta x \, {\rm e}^{-2\, n\, \gamma \, \Delta t} \, |\varepsilon_{j,n}|^2 
\le C \, \Delta t^2 \, \|u_0\|_{H^2(\RR^+)}^2 +C \, \Delta t^3 \, \|u_0\|_{H^2(\RR^+)}^2 \, 
\sum_{n \geq 2\, k} \, {\rm e}^{-2\, n\, \gamma \, \Delta t} \\
\le C \, \left( 1+\dfrac{1}{\gamma} \right) \, \Delta t^2 \, \|u_0\|_{H^2(\RR^+)}^2 \, .
\end{multline*}
\end{proof}

\begin{remark}
If we had not included the boundary layer corrector $u^{\rm bl,1}$ in the approximate solution, the right hand side 
in the error estimate \eqref{eq:internalerrorL2L2gam} would have been of the form $\Delta t \, \|u_0\|_{H^2(\RR^+)}^2$ 
instead of $\Delta t^2 \, \|u_0\|_{H^2(\RR^+)}^2$, which would have not been sufficient to derive 
\eqref{eq:semigroupestimate} because there is a loss of a factor $\Delta t$ in the derivation of the estimate 
\eqref{semigroup1} below.
\end{remark}

\subsection{The semigroup estimate}

We now prove Theorem~\ref{thm:stab1step}. We apply the main result\footnote{As a matter of fact, the main result 
of \cite{goldberg-tadmor} requires more restrictive conditions than Assumption \ref{as:stability}, but the extension of 
the result of \cite{goldberg-tadmor} to numerical schemes that satisfy Assumption \ref{as:stability} was performed in 
\cite{jfcnotes}.} of \cite{goldberg-tadmor} which states that the numerical scheme \eqref{eq:error} is strongly stable 
in the sense of \cite{gks}. In other words, there exists a constant $C>0$, that is independent of the parameter 
$\gamma>0$, such that there holds:
\begin{multline}
\label{eq:GKSestimate}
\dfrac{\gamma}{1+\gamma \, \Delta t} \, \sum_{n \geq 0} \, \sum_{j\geq 0} \Delta t \, \Delta x \, 
{\rm e}^{-2\, n\, \gamma \, \Delta t} \, |e_j^n|^2 +\sum_{n \geq 0} \, \sum_{j=0}^{r+p-1} \Delta t \, 
{\rm e}^{-2\, n\, \gamma \, \Delta t} \, |e_j^n|^2 \\
\le C\, \left( \dfrac{1+\gamma \, \Delta t}{\gamma} \, \sum_{n \geq k} \, \sum_{j\geq r} \Delta t \, \Delta x \, 
{\rm e}^{-2\, n\, \gamma \, \Delta t} \, |\varepsilon_j^n|^2 
+\sum_{n \geq k} \, \sum_{j=0}^{r-1} \Delta t \, {\rm e}^{-2\, n\, \gamma \, \Delta t} \, |\eta_j^n|^2 \right) \\
\le C \, \Delta t^2 \, \| u_0 \|^2_{H^2(\RR^+)} \, \left( \dfrac{\gamma \Delta t +1}{\gamma} \, 
\left( 1+\dfrac{1}{\gamma} \right) +1 \right) \, ,
\end{multline}
where we have used Proposition \ref{prop:errortermsgamm} to derive the second inequality in \eqref{eq:GKSestimate}. 
We choose $\gamma=\Delta t^\mu$, with $\mu \in [0,1/3]$. We thus derive from \eqref{eq:GKSestimate} the bound
$$
\sum_{n \geq 0} \Delta t \, {\rm e}^{-2\, n\, \Delta t^{1+\mu}} \, \sum_{j\geq 0} \Delta x \, |e_j^n|^2 
\le C \, \Delta t^{2-3\, \mu} \, \| u_0 \|^2_{H^2(\RR^+)} \, .
$$
In particular, a very crude lower bound for the left hand side gives
\begin{equation}
\label{semigroup1}
\sup_{n \geq 0} \, {\rm e}^{-2\, n \, \Delta t^{1+\mu}} \, \sum_{j \geq 0} \Delta x \, |e_j^n|^2 
\le C \, \Delta t^{1-3\, \mu} \, \| u_0 \|^2_{H^2(\RR^+)} \, .
\end{equation}

The semigroup estimate \eqref{semigroup1} yields the bound
$$
\forall \, n \in \NN \, ,\quad 
\sum_{j \geq 0} \Delta x \, |u_j^n|^2 \le 2\, \sum_{j \geq 0} \Delta x \, |u^{\rm app}_{j,n}|^2 
+C\, {\rm e}^{2\, n\, \Delta t^{1+\mu}} \, \Delta t^{1-3\, \mu} \, \| u_0 \|^2_{H^2(\RR^+)} \, ,
$$
with a constant $C$ that is uniform with respect to all the parameters. We now derive a semigroup estimate 
for the approximate solution $u^{\rm app}$. In the case of an incoming transport equation ($a>0$), we have
$$
u^{\rm app}_{j,n} =u^{\rm int}_{j,n} =\dfrac{1}{\Delta x} \, \int_{x_j}^{x_{j+1}} u_0(x-a\, t^n)\, {\rm d}x \, ,
$$
for all $j,n \in \NN$ (recall that $u_0$ vanishes on $\RR^-$). In particular, the Cauchy-Schwarz inequality yields
$$
\sum_{j \geq 0} \Delta x \, |u^{\rm app}_{j,n}|^2 \le \| u_0 \|_{L^2(\RR^+)}^2 \, ,
$$
and we get
$$
\forall \, n \in \NN \, ,\quad 
\sum_{j \geq 0} \Delta x \, |u_j^n|^2 \le 2\, \| u_0 \|_{L^2(\RR^+)}^2 
+C \, \Delta t^{1-3\, \mu} \, {\rm e}^{2\, n\, \Delta t^{1+\mu}} \, \| u_0 \|^2_{H^2(\RR^+)} \, ,
$$
which gives \eqref{eq:semigroupestimate}. We now consider the case of an outgoing transport equation ($a<0$) 
and derive a semigroup estimate for the approximate solution $u^{\rm app}$. We still have
$$
\sum_{j \geq 0} \Delta x \, |u^{\rm int}_{j,n}|^2 \le \| u_0 \|_{L^2(\RR^+)}^2 \, ,
$$
and we thus focus on the semigroup estimate for the boundary layer profile and corrector. Let us first 
consider the boundary layer profile $u^{\rm bl,0}$, for which we have
$$
\begin{aligned}
\sup_{n \ge 0} \, \sum_{j \geq 0} \Delta x \, |u^{\rm bl,0}_{j,n}|^2 
= \sup_{n \ge k} \, \sum_{j \geq 0} \Delta x \, |u^{\rm bl,0}_{j,n}|^2 
&= \sup_{n \ge k} \, \Delta x \, |u_n^{\rm tr}|^2 \, \sum_{j \geq 0} w_j^2 \\
&\le \sup_{n \ge k} \, \dfrac{C}{\Delta t} \, \left| \int_{t^n}^{t^{n+1}} u_0(|a| \, t) \, {\rm d}t \right|^2 
\le C\, \| u_0 \|_{L^2(\RR^+)}^2 \, .
\end{aligned}
$$
We now deal with the first boundary layer corrector $\Delta x \, u^{\rm bl,1}$, for which we have
$$
\begin{aligned}
\sup_{n \ge 0} \, \sum_{j \geq 0} \Delta x \, |\Delta x \, u^{\rm bl,1}_{j,n}|^2 
= \sup_{n \ge k} \, \sum_{j \geq 0} \Delta x^3 \, |u^{\rm bl,1}_{j,n}|^2 
&= \sup_{n \ge k} \, C\, \Delta x \, \left| \sum_{\sigma=0}^k \alpha_\sigma \, u^{\rm tr}_{n+\sigma} \right|^2 
\, \sum_{j \geq 0} \widetilde{w}_j^2 \\
&\le C\, \Delta t \, \sup_{n \ge k} \, \sum_{\sigma=0}^k |u^{\rm tr}_{n+\sigma}|^2 \le C \, \| u_0 \|_{L^2(\RR^+)}^2 \, .
\end{aligned}
$$
As in the incoming case, we have thus derived the bound
$$
\sum_{j \geq 0} \Delta x \, |u^{\rm app}_{j,n}|^2 \le C \, \| u_0 \|_{L^2(\RR^+)}^2 \, ,
$$
and we get \eqref{eq:semigroupestimate} accordingly.

\section{Example and counterexample}
\label{sect4}

\subsection{A 4 time-step 5 point centered scheme}

As a first numerical illustration of the above results in the case of an outgoing velocity $a=-1$, we consider the 
following numerical scheme. The time-stepping is solved using the 3rd order explicit Adams-Bashforth method, 
so that assumption~\ref{as:consistencymultistep} is satisfied. The space discretization of the advection term 
$a \, \partial_x u$ is based on a centered five-point approximation supplemented with a fourth order stabilizing 
dissipative term:
\begin{equation}
\label{eq:fiveVarpoints}
\begin{aligned}
& u_{j}^{n+1} = u_{j}^{n} - \lambda \left(\dfrac{23}{12}f_{j}^{n}- \dfrac{16}{12}f_{j}^{n-1}+\dfrac{5}{12}f_{j}^{n-2}\right) \, ,\\
& f_{j}^{n} := a \, \dfrac{-u_{j+2}^{n}+8u_{j+1}^{n}-8u_{j-1}^{n}+u_{j-2}^{n}}{12} 
- \dfrac{-u_{j+2}^n+4u_{j+1}^{n}-6u_{j}^{n}+4u_{j-1}^{n}-u_{j-2}^n}{24} \, .
\end{aligned}
\end{equation}
As we will show with numerical experiments, this scheme displays numerical boundary layers when combined with 
Dirichlet boundary conditions. Let us compute $\calA$:
\[
\calA(z) = \dfrac{a}{12}(-z^2+8z-8z^{-1}+z^{-2}) - \dfrac{1}{24}(-z^2+4z-6+4z^{-1}-z^{-2}),
\]
from which we get $\calA(1)=0$ and $\calA'(1)=a$ and the space discretization satisfies assumption~\ref{as:consistency}. 
Moreover, for any $\theta\in\RR$, one gets
\[
\calA({\rm e}^{i\theta})=-a\dfrac{i}{6}(\sin(2\theta)-8\sin(\theta))-\dfrac{1}{12}(-\cos(2\theta)+4\cos(\theta)-3)\, , 
\textrm{ and } \Re (\calA(e^{i\theta}))=\dfrac{2}{3}\sin^4\left(\dfrac{\theta}{2}\right),
\]
Therefore the only root of $\calA(e^{i\theta})$ in $[-\pi,\pi]$ is $\theta=0$. This ensures that assumption~\ref{as:circle} is 
satisfied. Figure~\ref{fig:stabAB3Var} below pictures the closed curve $\left\{-\lambda\calA(e^{i\eta}),\eta\in\mathbb{R} 
\right\}$ for the choice $\lambda=0.4$ (blue curve), together with the stability domain of the time integrator (red dashed 
curve) ; see \cite{hw,hnw} for details. Let us observe that the stability assumption for the Cauchy problem~\ref{as:stability} 
is satisfied.

\begin{figure}[h!]
\includegraphics[scale=0.4]{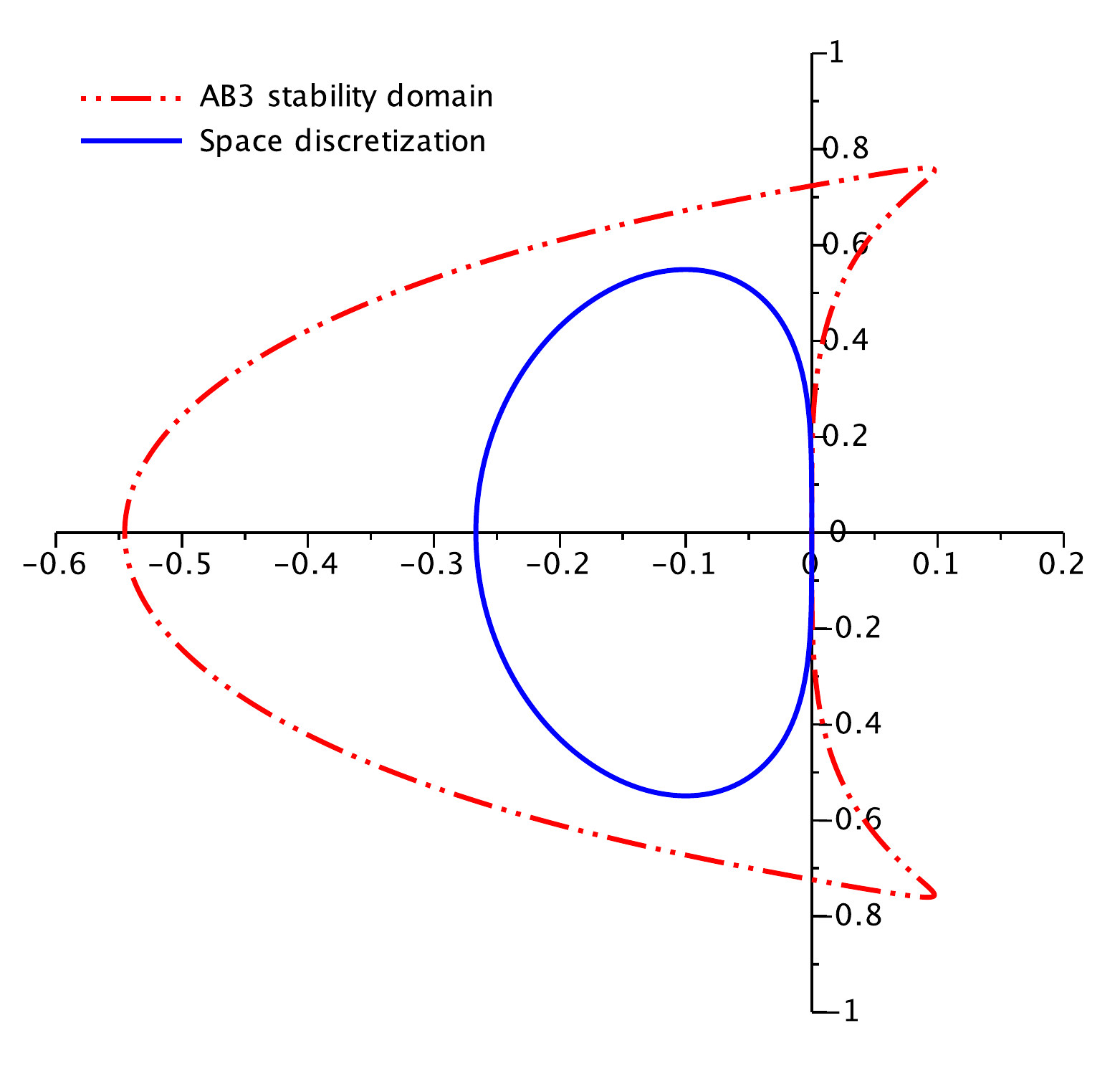}
\caption{Verification of the stability assumption~\ref{as:stability} for the 3rd order scheme \eqref{eq:fiveVarpoints}.}
\label{fig:stabAB3Var}
\end{figure}

\bigskip
The numerical test case concerns the following initial condition:
\[
u_0(x) =  {\rm e}^{-100(x-0.5)^2},\quad x\in [0,1] \, ,
\]
together with homogeneous Dirichlet conditions at both left and right boundaries (with no significant effect arising from 
the right boundary due to the incoming situation with $a=-1$ and the absence of boundary layer at $x=1$). We compute 
the solution on $N=216$ uniformly spaced grid cells, until time $T=0.5$. At this time, the initial bump crosses the left 
boundary with the highest strength. As expected, the numerical solution $(u_j^n)$ develops some boundary layer in 
the neighborhood of $x=0$ due to the incompatibility of the homogeneous Dirichlet condition $u_j^n=0$, $0\leq j\leq r-1$, 
with the effective trace of the solution $u^{\rm int}(0^+,T)=1$. We then observe on Figure~\ref{fig:AB3state} an oscillating 
pattern that does not disappear as $\Delta x$ tends to $0$. The two roots of $\calA$ in $\DD \setminus \{ 0 \}$ are real 
and distinct; one of them equals approximately $0.0809$ and therefore belongs to $(0,1)$, while the second one equals 
approximately $- 0.6595$ and therefore belongs to $(-1,0)$, which gives rise to the oscillations in the boundary layer.

\begin{figure}[h!]
\includegraphics[scale=0.4]{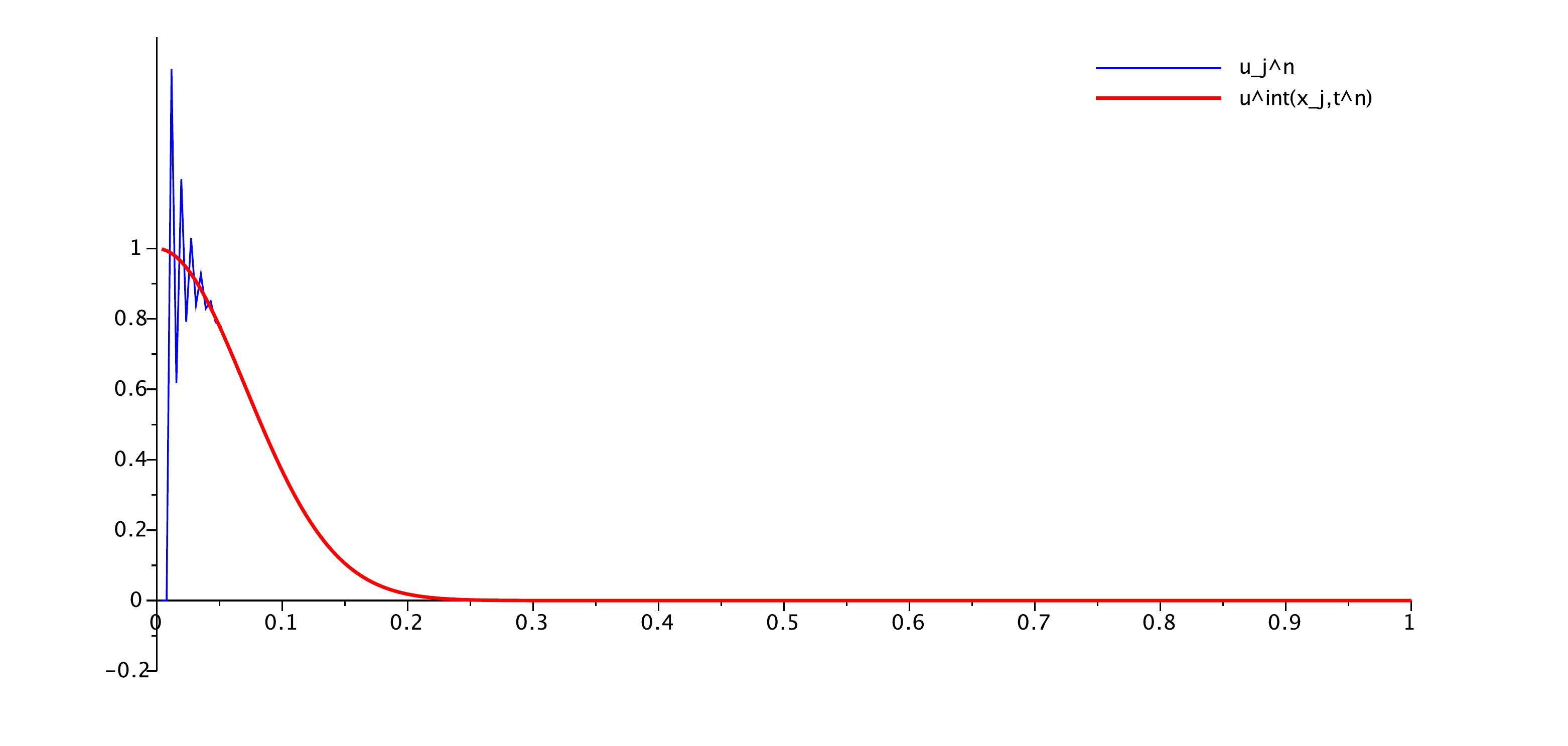}
\caption{Numerical solution and exact solution at time T=0.5 (216 grid points).}
\label{fig:AB3state}
\end{figure}

The main term of the boundary layer expansion is a linear combination of two geometric sequences generated by the 
roots of the equation $\calA(z)=0$ in $\mathbb{D} \setminus \{0\}$ (see Lemma~\ref{lm:roots}). In the present case, 
we obtain numerically $z_1 \simeq - 0.6595$ and $z_2 \simeq 0.0809$. The precise boundary layer expansion 
$u^{\rm bl, 0}(j,T) + \Delta x \, u^{\rm bl, 1}(j,T)$ is depicted with crosses on the left picture of Figure~\ref{fig:AB3diff} 
for the first 20 grid cells. Notice that it depends only on the trace of the solution at the considered time $u^{\rm tr}_n$ 
and on the discrete in time derivative of this trace, through $u^{\rm bl, 1}$. It fits quite well the difference between the 
numerical solution and the exact one $u_j^n-u^{\rm int}(x_j,T)$. On the right picture of Figure~\ref{fig:AB3diff} is 
represented the error in this boundary layer expansion $u_j^n-\left[ u^{\rm int}(x_j,T) +u^{\rm bl, 0}(j,T) +\Delta x \, 
u^{\rm bl, 1}(j,T)\right]$ in the first 50 grid cells.

\begin{figure}[h!]
\includegraphics[scale=0.4]{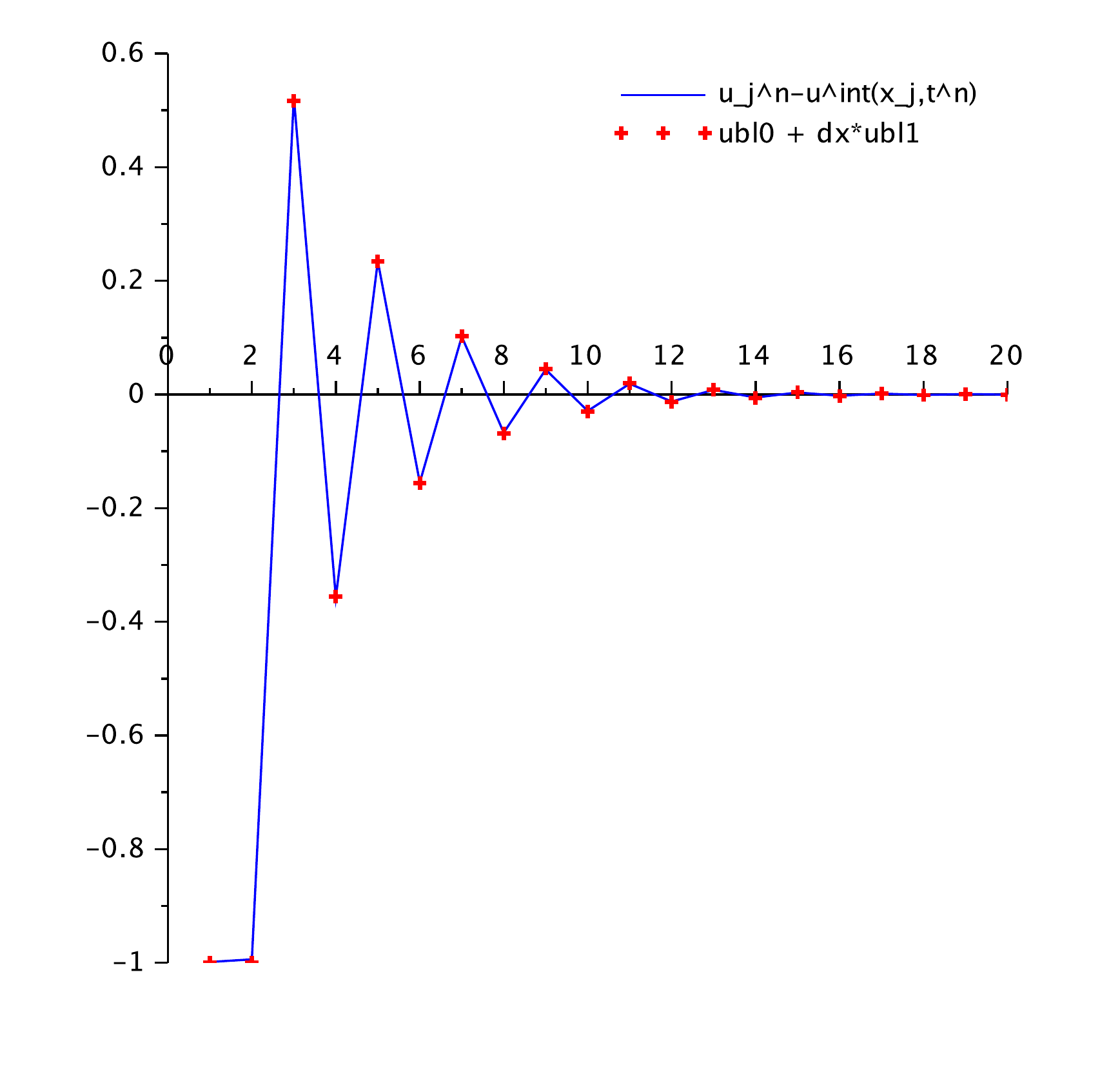}
\includegraphics[scale=0.4]{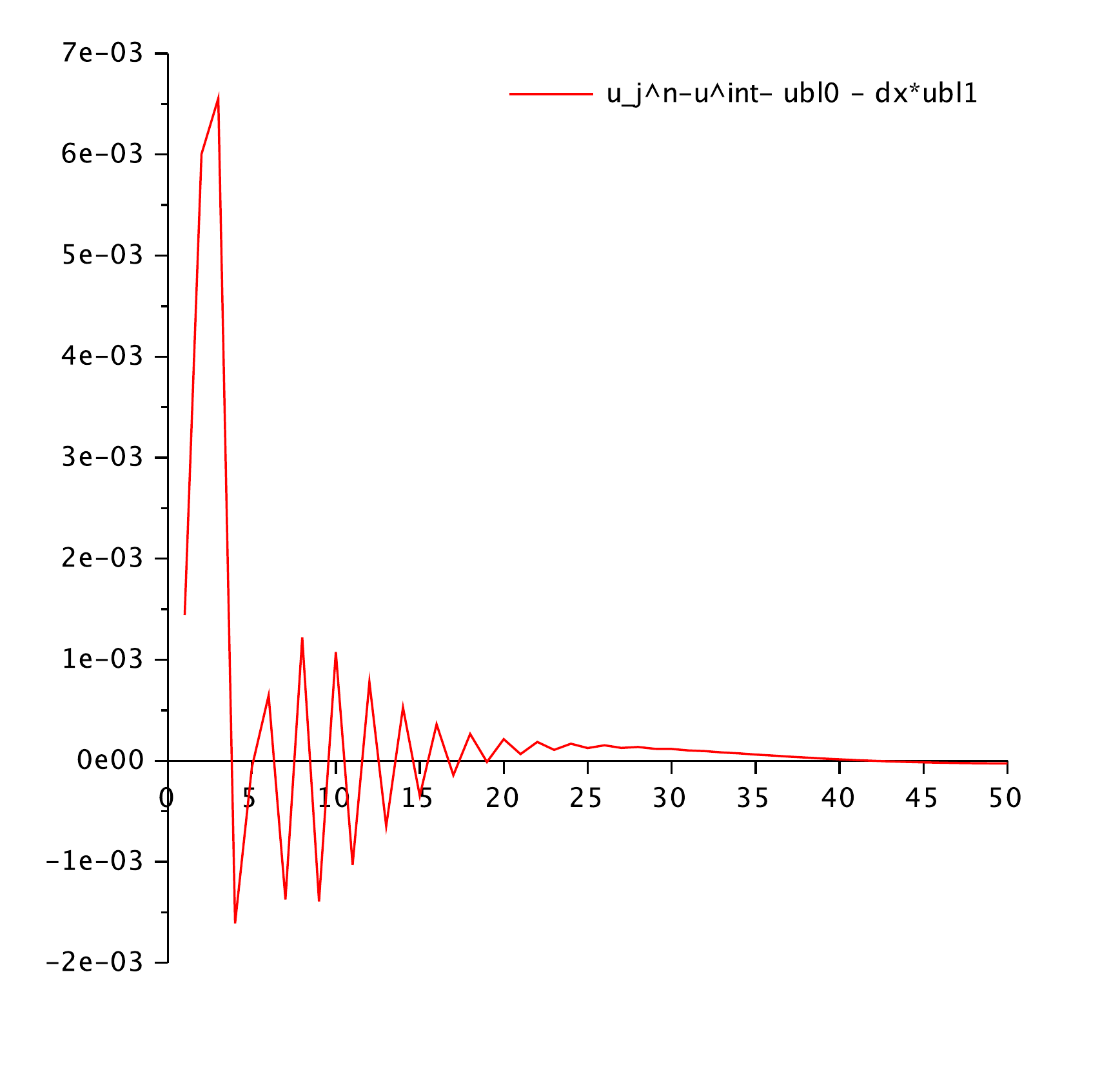}
\caption{Boundary layer expansion at time T=0.5 (216 grid points)}
\label{fig:AB3diff}
\end{figure}

The scheme~\eqref{eq:fiveVarpoints} is third order in time and space accurate. We now consider the effective accuracy 
of this scheme for the IBVP problem by computing the $\ell^2([0,1])$ error at a final given time for successive values of 
$2^M$ grid points, $5\leq M \leq 12$. More precisely, given a time $T>0$, we compute the following two quantities, where 
$n=N_T$ is the first integer such that $N_T \Delta t \geq  T$:
\[
\left( \sum_{j=0}^{2^M} \Delta x \, \left|u_j^n-u^{\rm int}(x_j,t^n)\right|^2 \right)^{1/2} \, , \quad \textrm{and} \quad 
\left( \sum_{j=0}^{2^M} \Delta x \, \left|u_j^n-u^{\rm app}(x_j,t^n)\right|^2 \right)^{1/2} \, .
\]
At a first time $T=0.125$ at which no significant boundary layer has appeared at $x=0$, the convergence of both quantities 
occur with order 3, see Figure~\ref{fig:Varconvergence} on the left. For very thin grids, one observes however a slight loss 
of accuracy when computing the usual numerical error. It corresponds to the presence of a very small boundary layer that 
deteriorates the effective order of accuracy.

At a later time $T=0.4$ at which the boundary layer is sufficiently high to affect the convergence, the usual numerical error 
is strongly increased and the apparent order of accuracy is severely damaged: in Figure~\ref{fig:Varconvergence} on the 
right, we observe a numerical accuracy of order $0.5$ for the usual numerical error, and of $1.5$ for the error in the 
boundary layer expansion.

\begin{figure}[h!]
\includegraphics[scale=0.45]{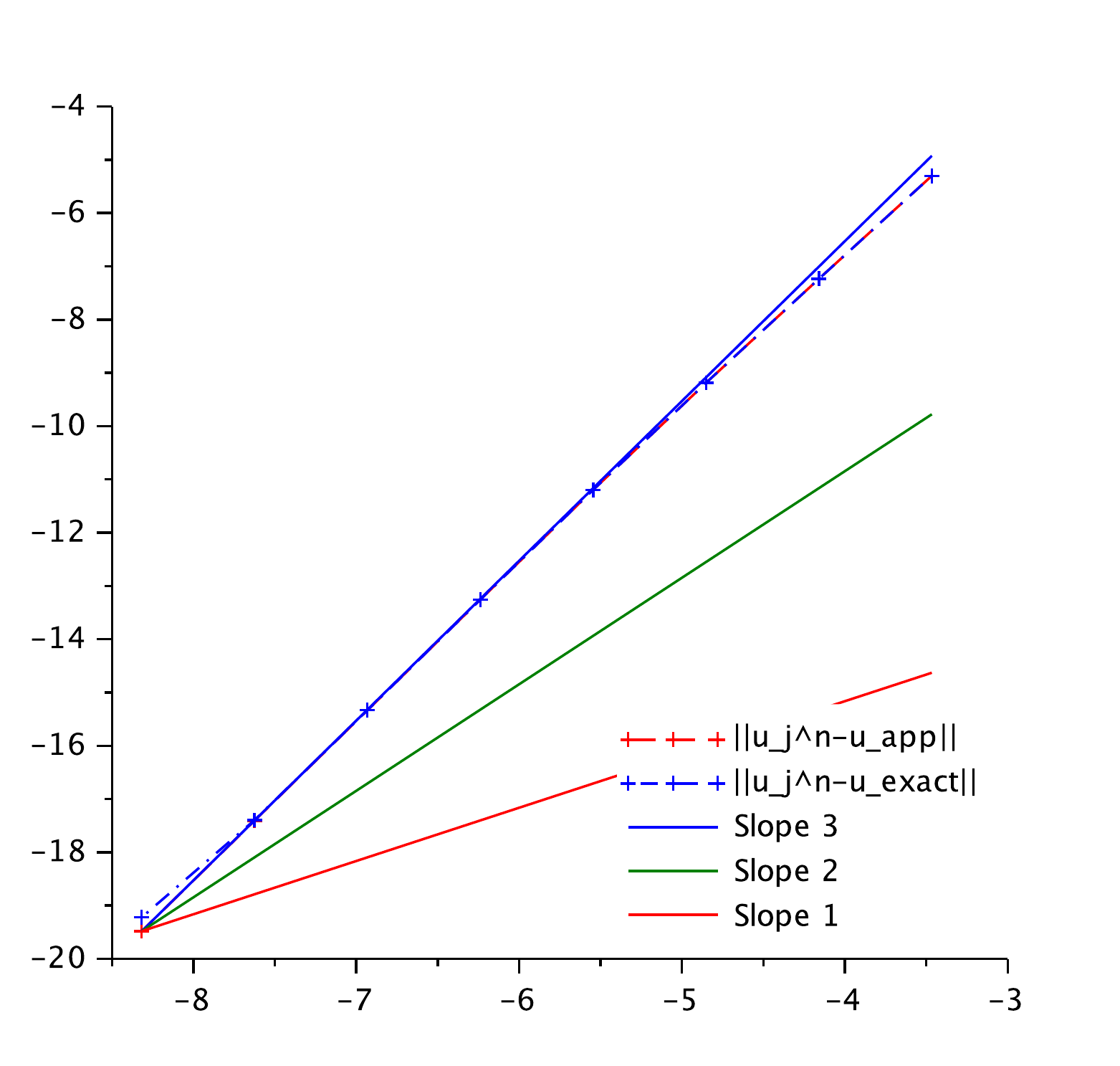}
\includegraphics[scale=0.45]{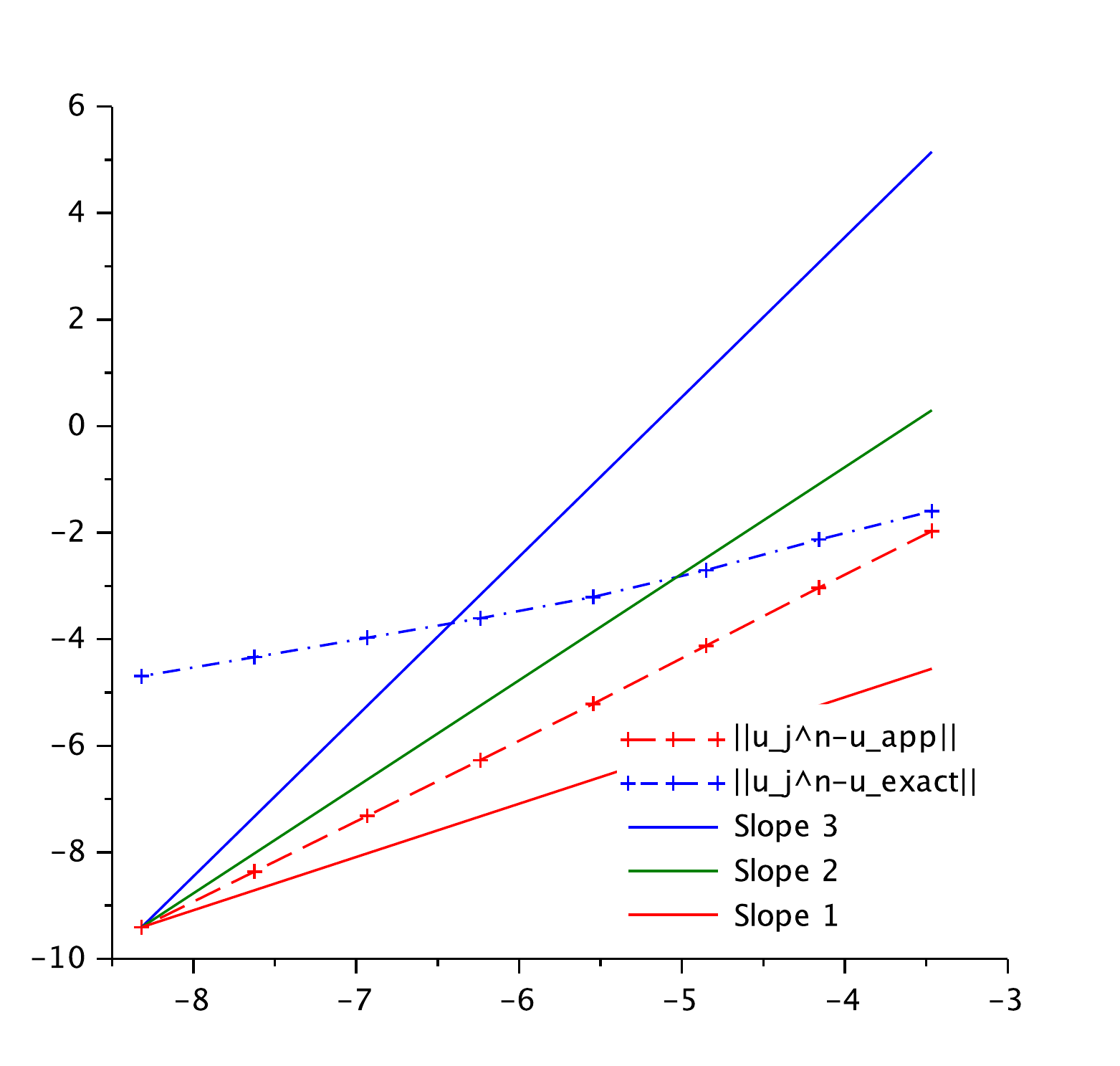}
\caption{Convergence in log/log scale. Solution at time $T=0.125$ with no significant boundary layer (left) / at time 
$T=0.4$ with an important boundary layer (right).}
\label{fig:Varconvergence}
\end{figure}

\subsection{The leap-frog scheme}

We now consider the usual three-time step leap-frog scheme, with a three point stencil in space:
\begin{equation}
\label{eq:Leapfrog}
\dfrac{u^{n+1}_j-u^{n-1}_j}{2\Delta t} +a \, \dfrac{u^{n}_{j+1}-u^n_{j-1}}{2\Delta x}=0 \, .
\end{equation}
The scheme corresponds to the so-called Nystr\"om method of order $2$ (also called the mid-point formula) 
combined with the center differentiation formula for the space discretization. The corresponding function 
$\calA$ equals $z-z^{-1}$, and therefore vanishes at $-1$. Assumption~\ref{as:circle} is no longer satisfied 
and Figure \ref{fig:SM} below illustrates that the failure of Assumption~\ref{as:circle} gives rise to a completely 
different behavior. Namely, we compute the numerical solution for \eqref{eq:Leapfrog} with $a=-1$ and 
homogeneous Dirichlet boundary conditions at different time levels, for the same kind of bump initial data. 
As the bump crosses the left boundary, a highly oscillatory wave packet emerges from the boundary and 
propagates with velocity $+1$ towards the right. The envelope of this wave packet is exactly the one of the 
initial condition, see Figure~\ref{fig:SM}. The latter phenomenon has long been identified of course, see, 
e. g., \cite{trefethen}.

\begin{figure}[h!]
\includegraphics[trim=20 20 20 20,clip,scale=0.35]{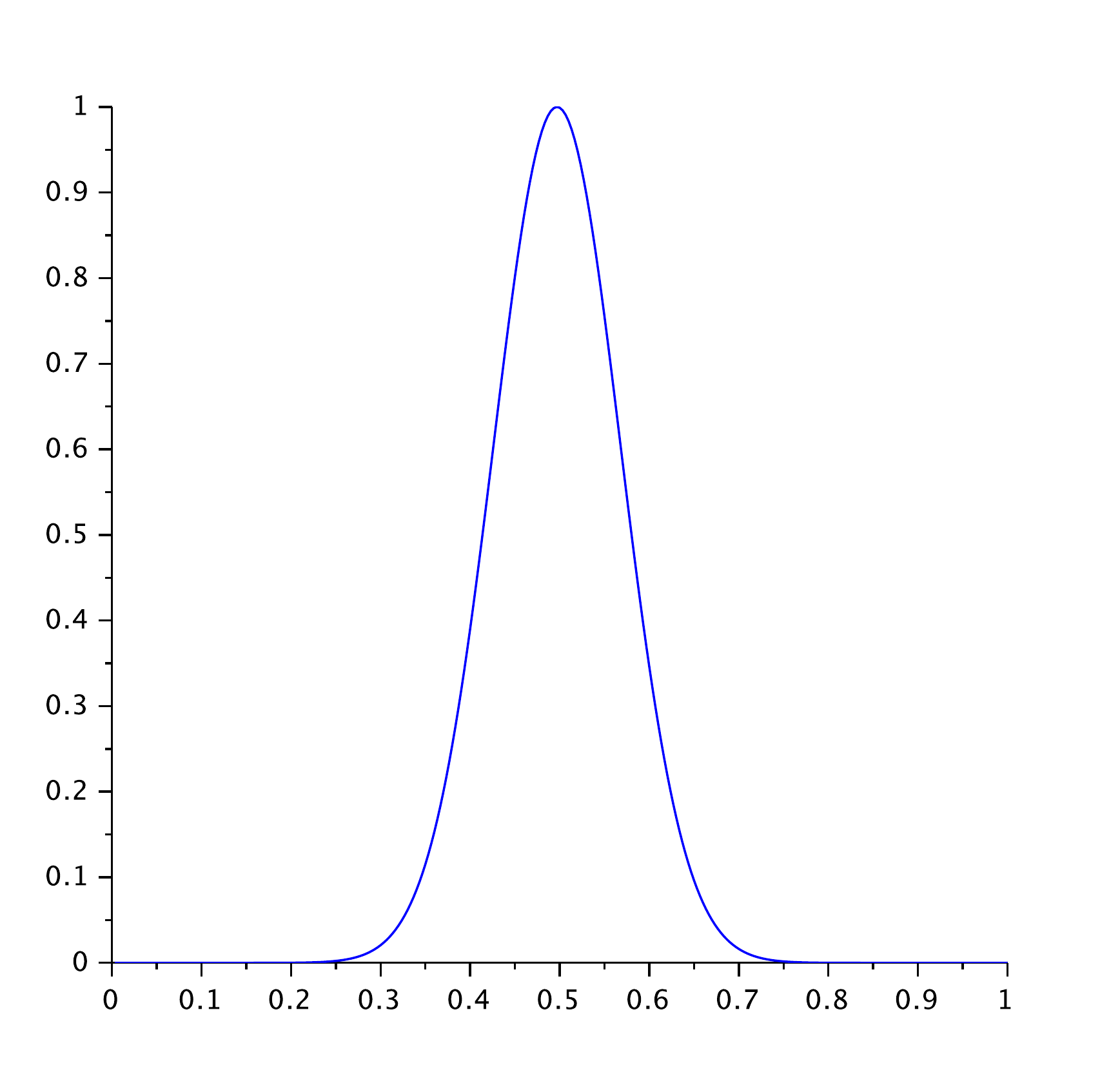}
\includegraphics[trim=20 20 20 20,clip,scale=0.35]{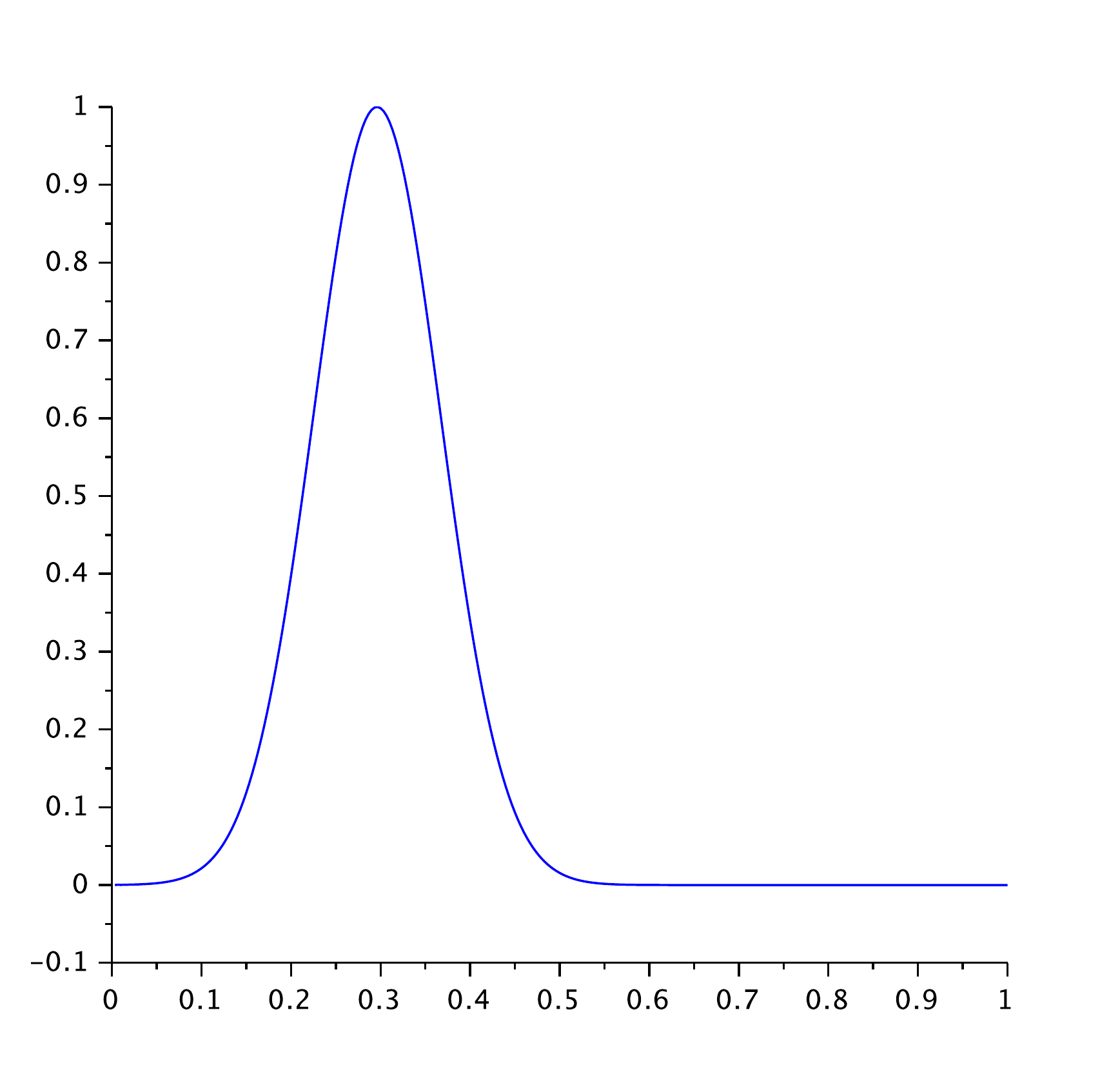}

\includegraphics[trim=20 20 20 20,clip,scale=0.35]{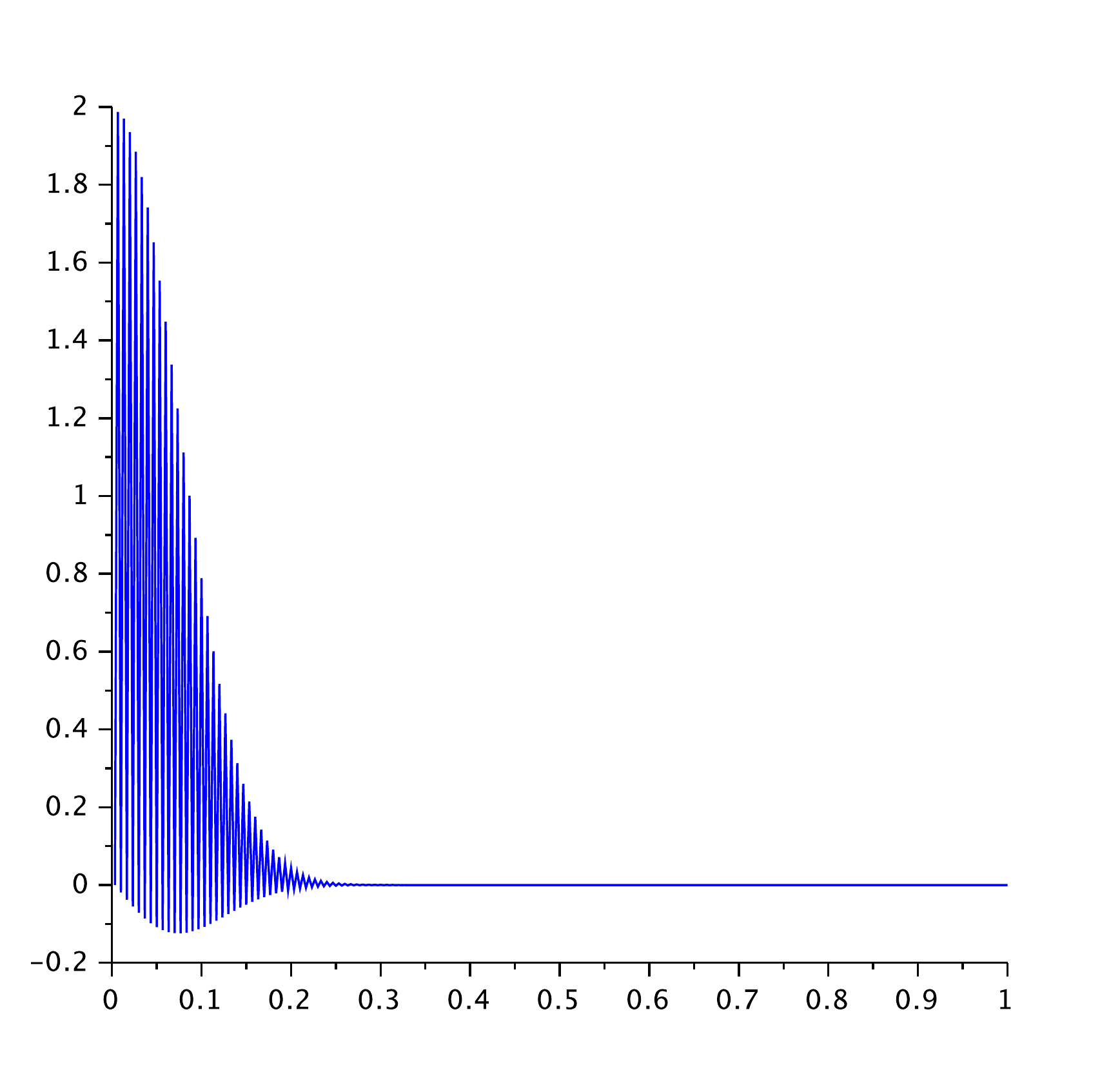}
\includegraphics[trim=20 20 20 20,clip,scale=0.35]{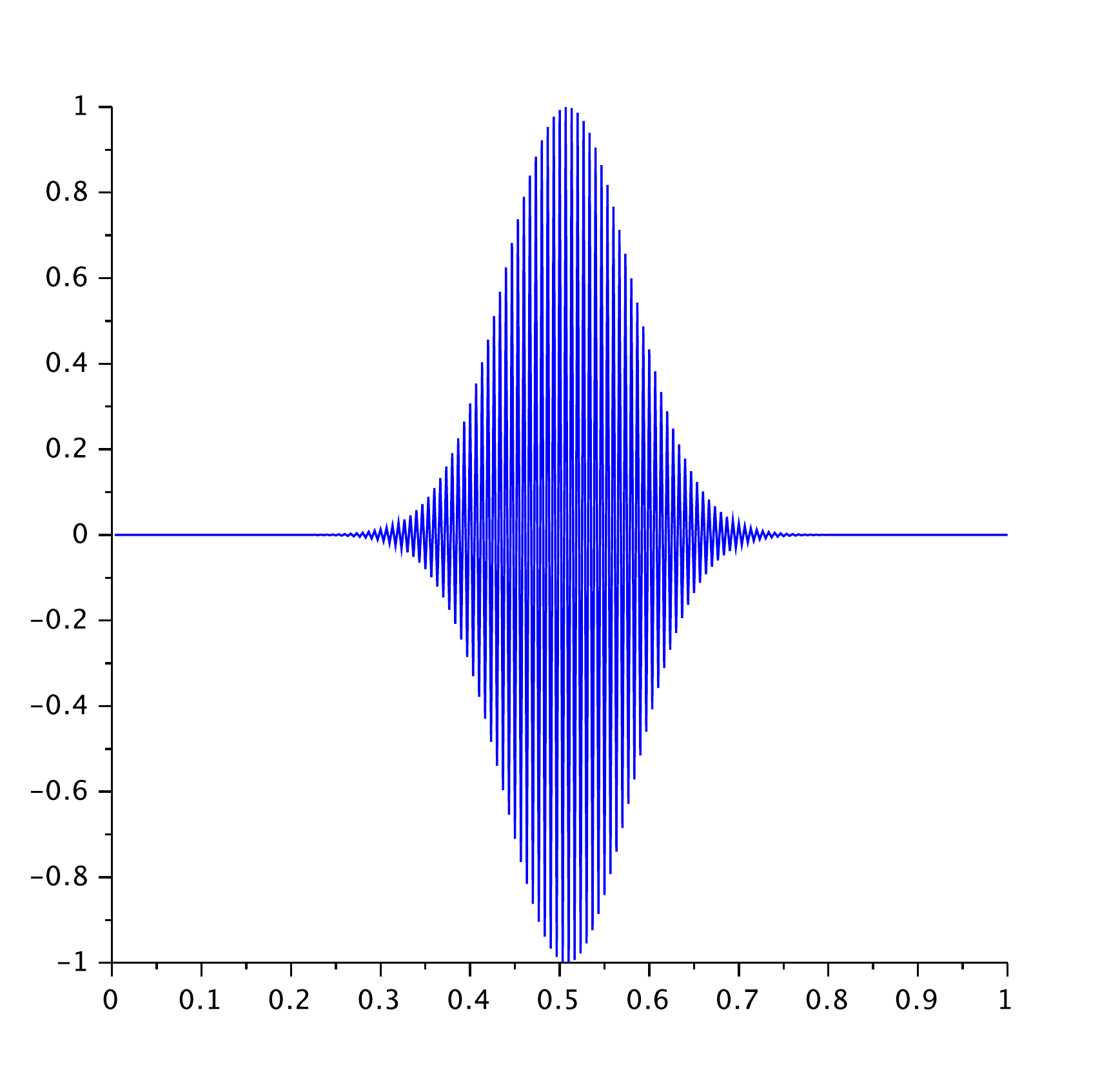}
\caption{Leap-frog scheme, solution at time $T=0$, $T=0.2$, $T=0.5$ and $T=1$}
\label{fig:SM}
\end{figure}

\bibliographystyle{alpha}
\bibliography{notes}
\end{document}